\documentclass[11pt]{article}
\pdfoutput=1
\usepackage[UKenglish]{isodate}
\usepackage[T1]{fontenc}
\usepackage[noBBpl]{mathpazo}
\usepackage{microtype}
\usepackage{multicol}
\cleanlookdateon
\RequirePackage[a4paper, left=2.8cm, right=2.8cm, top=2.7cm, bottom=2.7cm]{geometry}
\usepackage{amsmath, amssymb, amsthm, amsfonts, dsfont, zlmtt, float, url, xpatch, stmaryrd}

\newcommand{\e}{\varepsilon}
\newcommand{\p}{\varphi}
\newcommand{\C}{\mathcal{C}}
\newcommand{\N}{\mathcal{N}}
\renewcommand{\S}{\mathcal{S}}
\newcommand{\F}{\mathbb{F}}
\newcommand{\FF}{\overline{\F}}
\renewcommand{\:}{\colon}
\renewcommand{\mod}[1]{\mathrm{ \ } (\mathrm{mod\ } #1)}
\newcommand{\im}{\operatorname{im}}
\newcommand{\<}{\langle}
\renewcommand{\>}{\rangle}
\renewcommand{\leq}{\leqslant}
\newcommand{\leqn}{\trianglelefteqslant}
\renewcommand{\geq}{\geqslant}

\newcommand{\soc}{\operatorname{soc}}
\newcommand{\Aut}{\operatorname{Aut}}

\newcommand{\Out}{\operatorname{Out}}
\newcommand{\Sym}{\operatorname{Sym}}

\newcommand{\Frat}{\operatorname{Frat}}
\newcommand{\End}{\operatorname{End}}
\newcommand{\fix}{\operatorname{fix}}
\newcommand{\AGL}{\operatorname{AGL}}

\newcommand{\GL}{\operatorname{GL}}
\newcommand{\PSL}{\operatorname{PSL}}
\newcommand{\PGL}{\operatorname{PGL}}
\newcommand{\Sp}{\operatorname{Sp}}
\newcommand{\PSp}{\operatorname{PSp}}

\newcommand{\PSU}{\operatorname{PSU}}

\newcommand{\POm}{\operatorname{P}\!\Omega}
\newcommand{\Spin}{\operatorname{Spin}}

\newtheoremstyle{shdefinition}{6pt}{3pt}{}{}{\bfseries\boldmath}{.}{0.3em}{} 
\newtheoremstyle{shplain}{6pt}{3pt}{\itshape}{}{\bfseries\boldmath}{.}{0.3em}{} 
\theoremstyle{shdefinition}
\newtheorem{definition}{Definition}[section]
\newtheorem{shdefinition}{Definition}
\newtheorem{remark}[definition]{Remark}
\newtheorem{shremark}[shdefinition]{Remark}
\newtheorem{example}[definition]{Example}

\newtheorem*{notation*}{Notation}
\newtheorem*{acknowledgements*}{Acknowledgements}
\theoremstyle{shplain}
\newtheorem{theorem}[definition]{Theorem}
\newtheorem{shtheorem}[shdefinition]{Theorem}
\newtheorem*{theorem*}{Theorem}
\newtheorem{shconjecture}[shdefinition]{Conjecture}
\newtheorem{corollary}[definition]{Corollary}
\newtheorem{shcorollary}[shdefinition]{Corollary}
\newtheorem{proposition}[definition]{Proposition}
\newtheorem{lemma}[definition]{Lemma}
\newtheorem{hypothesis}[definition]{Hypothesis}
\newtheorem*{shtheoremstar}{Theorem~\ref{thm:main}*}

\RequirePackage{titlesec}
\titlespacing*{\section}{0pt}{\baselineskip}{0pt}
\titlespacing*{\subsection}{0pt}{0.66\baselineskip}{0pt}

\RequirePackage[shortlabels]{enumitem}
\setlist{leftmargin=0.8cm,topsep=0pt,itemsep=-2pt}
\setlist[enumerate]{label=\rm{(\roman*)}}
\numberwithin{equation}{section}
\RequirePackage[tableposition=top, figureposition=bottom, skip=0.5\baselineskip]{caption}

\renewenvironment{thebibliography}[1]
{ \begin{oldthebibliography}{#1}
  \setlength{\parskip}{0pt}
  \setlength{\itemsep}{2pt plus 0.3ex}
  \bgroup\footnotesize }
{ \egroup \end{oldthebibliography} }

\makeatletter
\renewenvironment{proof}[1][\proofname]{\par
  \pushQED{\qed}%
  \normalfont
  \topsep2pt \partopsep1pt 
  \trivlist
  \item[\hskip\labelsep
        \itshape
    #1\@addpunct{.}]\ignorespaces
}{%
  \popQED\endtrivlist\@endpefalse
  \addvspace{6pt plus 6pt}
}
\makeatother

\makeatletter
\g@addto@macro\normalsize{%
  \setlength\abovedisplayskip{0.4\baselineskip plus 0.4\baselineskip}
  \setlength\belowdisplayskip{0.4\baselineskip plus 0.4\baselineskip}
  \setlength\abovedisplayshortskip{-0.3\baselineskip}
  \setlength\belowdisplayshortskip{0.4\baselineskip plus 0.4\baselineskip}
}
\makeatother

\makeatletter\def\blfootnote{\gdef\@thefnmark{}\@footnotetext} \makeatother

\begin{document}

\begin{center} 
{\LARGE \textbf{Derangements in intransitive groups}} \\[11pt]
{\Large David Ellis \& Scott Harper}                  \\[22pt]
\end{center}

\begin{center}
\begin{minipage}{0.866\textwidth}
\small
Let $G$ be a nontrivial permutation group of degree $n$. If $G$ is transitive, then a theorem of Jordan states that $G$ has a derangement. Equivalently, a finite group is never the union of conjugates of a proper subgroup. If $G$ is intransitive, then $G$ may fail to have a derangement, and this can happen even if $G$ has only two orbits, both of which have size $(1/2+o(1))n$. However, we conjecture that if $G$ has two orbits of size exactly $n/2$ then $G$ does have a derangement, and we prove this conjecture when $G$ acts primitively on at least one of the orbits. Equivalently, we conjecture that a finite group is never the union of conjugates of two proper subgroups of the same order, and we prove this conjecture when at least one of the subgroups is maximal. (Feldman also implicitly raised this conjecture on StackExchange.) We also prove the conjecture for soluble groups, almost simple groups and groups of order at most 50000, and we reduce the conjecture to perfect groups. Along the way, we prove a linear variant on Isbell's conjecture regarding derangements of prime-power order, and we highlight connections with intersecting families of permutations and roots of polynomials modulo primes. 
\end{minipage}
\end{center}

\section{Introduction} \label{s:intro}

Let $G \leq \Sym(n)$ be a nontrivial permutation group. By an 1872 theorem of Jordan \cite{ref:Jordan72}, if $G$ is transitive then it contains a \emph{derangement}, i.e. a fixed-point-free permutation. Equivalently, a finite group is never the union of the conjugates of a proper subgroup. As highlighted by Serre \cite{ref:Serre03}, this elementary group theoretic fact has consequences for number theory, topology and character theory. There are now numerous generalisations of this classical result, and we refer the reader to the introductory chapter of \cite{ref:BurnessGiudici16} for an overview.

The hypothesis that $G$ is finite is necessary for Jordan's theorem. Indeed, every element of $G = \GL_n(\mathbb{C})$ is conjugate to an upper triangular matrix, so $G$ is the union of conjugates of the subgroup of $G$ consisting of all upper triangular matrices. More generally, a connected linear algebraic group over an algebraically closed field is the union of conjugates of a Borel subgroup. Every transitive action of $G$ has a derangement if and only if $G$ has an invariable generating set, and much recent work has focussed on determining which infinite groups do, beginning with Kantor, Lubotzky and Shalev in \cite{ref:KantorLubotzkyShalev15}.

The hypothesis that $G$ is transitive is also necessary. For a trivial example, every element of $G = \Sym(n-1) \leq \Sym(n)$ fixes the point $n$, so $G$ has no derangements. Moreover, when $G \leq \Sym(n)$ has no derangements, the size of the smallest orbit of $G$ can be arbitrarily close to $\frac{n}{2}$ (see Remark~\ref{rem:coverings}(ii) below), but we conjecture that it can never equal $\frac{n}{2}$.

\begin{shconjecture} \label{conj:main}
Let $G \leq \Sym(n)$ have two orbits of size $\frac{n}{2} > 1$. Then $G$ contains a derangement.
\end{shconjecture}

Conjecture~\ref{conj:main} asserts that no finite group is the union of conjugates of two proper subgroups of the same order (see Remark~\ref{rem:coverings}). Feldman \cite{ref:Feldman14} also asked this on StackExchange.

We now state our main theorem, which implies an important special case of Conjecture~\ref{conj:main}.

\begin{shtheorem} \label{thm:main}
Let $G \leq \Sym(n)$ have exactly two orbits $\Omega_1$ and $\Omega_2$, both of which are nontrivial. Assume that $|\Omega_1|$ divides $|\Omega_2|$ and that $G$ acts primitively on $\Omega_2$. Then $G$ contains a derangement.
\end{shtheorem}

Equivalently, Theorem~\ref{thm:main} asserts that a finite group is not the union of conjugates of a proper subgroup $H_1$ and a maximal subgroup $H_2$ such that $|H_2|$ divides $|H_1|$. The assumption that $G$ is primitive on $\Omega_2$ (equivalently, that $H_2$ is maximal) is necessary, see Remark~\ref{rem:reduction}\ref{rem:reduction_example}. 

We now state some corollaries of Theorem~\ref{thm:main}. The first, which is an immediate consequence of Theorem~\ref{thm:main}, establishes Conjecture~\ref{conj:main} in an important special case.

\begin{shcorollary} \label{cor:main}
Let $G \leq \Sym(n)$ have two orbits of size $\frac{n}{2} > 1$. Assume that $G$ acts primitively on at least one of its orbits. Then $G$ contains a derangement.
\end{shcorollary}

We give the short proof of our second corollary at the end of Section~\ref{s:proofs}.

\begin{shcorollary} \label{cor:prime_power}
Let $G \leq \Sym(n)$ have exactly two orbits $\Omega_1$ and $\Omega_2$, both of which are nontrivial. Assume that $|\Omega_1|$ and $|\Omega_2|$ are powers of the same prime. Then $G$ contains a derangement. In particular, Conjecture~\ref{conj:main} is true when $\frac{n}{2}$ is a prime power.
\end{shcorollary}

Our next main theorem gives even more evidence towards Conjecture~\ref{conj:main}.

\begin{shtheorem} \label{thm:summary}
Conjecture~\ref{conj:main} is true when any of the following hold:
\begin{enumerate}
\item $G$ is soluble
\item $G$ is almost simple
\item $|G| \leq 50000$.
\end{enumerate}
\end{shtheorem}

We also have the following reduction theorem.

\begin{shtheorem} \label{thm:reduction}
To prove Conjecture~\ref{conj:main} it suffices to assume that $G$ is perfect.
\end{shtheorem}

Since this paper first appeared on the arxiv, Lee, Popiel and Verret \cite{ref:LeePopielVerret} proved that Conjecture~\ref{conj:main} also holds when $\frac{n}{2}$ is a product of two primes and when $\frac{n}{2} \leq 30$. 

Our main results and conjecture, which regard permutation groups, have interpretations in abstract group theory, graph theory, extremal combinatorics and algebraic number theory, as we explain in the following four remarks.

\begin{shremark} \label{rem:coverings}
Let $G \leq \Sym(n)$, let $\Omega_1$, \dots, $\Omega_k$ be the orbits of $G$ and let $H_i$ be the stabiliser of a point of $\Omega_i$ for each $1 \leq i \leq k$. Then $x \in G$ is a derangement if and only if $x \not\in \cup_{1 \leq i \leq k} \cup_{g \in G} H_i^g$, so $G$ has no derangements if and only if $G$ is the union of conjugates of $H_1$, \dots, $H_k$, in which case, $\{H_1, \dots, H_k\}$ is called a \emph{normal ($k$-)covering} of $G$. This observation gives the following.
\begin{enumerate}
\item Jordan's theorem is equivalent to the fact that a finite group is not the union of the conjugates of a proper subgroup.
\item A finite group can be the union of conjugates of two proper subgroups, and these subgroups can have arbitrarily close orders. For example, the affine group $\AGL_1(p) = \F_p{:}\F_p^\times$ is the union of the normal subgroup $\F_p$ and the conjugates of the complement $\F_p^\times$. There are also examples where both subgroups are core-free: if $m \geq 1$ and $q = 2^f$, then the symplectic group $\Sp_{2m}(q)$ is the union of conjugates of $H_1 = \mathrm{O}^+_{2m}(q)$ and $H_2 = \mathrm{O}^-_{2m}(q)$ (see \cite{ref:Dye79}), and here $|H_1|/|H_2| = (q^m-1)/(q^m+1)$.
\item Conjecture~\ref{conj:main} asserts that a finite group is not the union of conjugates of two proper subgroups of the same order, and Corollary~\ref{cor:main} proves this when at least one of the subgroups is maximal. This feels like the hardest case, but see the warning in Remark~\ref{rem:reduction}\ref{rem:reduction_warning}.
\item Conjecture~\ref{conj:main} vastly generalises the theorem that a finite group $G$ does not have a normal covering $\{H,H\alpha\}$ for a proper subgroup $H$ and an automorphism $\alpha$ of $G$. An argument of Jehne reduces this theorem to simple groups (see the proof of \cite[Theorem~5]{ref:Jehne77}) and simple groups were handled by Saxl \cite{ref:Saxl88} using the Classification of Finite Simple Groups. This result was the final part of the proof of the following theorem in algebraic number theory. Let $K/k$ be a quadratic extension of a number field $k$. Then $K/k$ is Kronecker equivalent to an extension $K'/k$ if and only if $K$ is $k$-isomorphic to $K'$. For further explanation, see the introduction to \cite{ref:Saxl88} and the survey \cite{ref:Klingen98}.
\end{enumerate}
\end{shremark}

\begin{shremark} \label{rem:graphs}
Permutation groups $G \leq \Sym(n)$ with two orbits of size $\frac{n}{2}$ arise naturally in graph theory. Let $\Gamma$ be a regular graph such that $\Aut(\Gamma)$ is transitive on edges. Then either $\Aut(\Gamma)$ is transitive on vertices, or $\Aut(\Gamma)$ has exactly two, equal-sized, orbits on vertices. In particular, Conjecture~\ref{conj:main} asserts that $\Aut(\Gamma)$ has a derangement. While this is open in general, it is known when $\Gamma$ is $3$- or $4$-regular by \cite[Corollary~1.4]{ref:GiudiciPotocnikVerret14} and when $\Gamma$ has fewer than $64$ vertices by a computational verification by Conder (such graphs are classified in \cite{ref:ConderVerret19}).
\end{shremark}

\begin{shremark} \label{rem:intersecting}
A set $S \subseteq \Sym(\Omega)$ is \emph{intersecting} if for all $x,y \in S$ there is $\omega \in \Omega$ with $\omega x = \omega y$. By analogy with the Erd\H{o}s--Ko--Rado theorem on intersecting families of subsets \cite{ref:ErdosKoRado61}, Frankl and Deza \cite{ref:FranklDeza77} proved that if $S \subseteq \Sym(\Omega)$ is intersecting, then $|S| \leq (|\Omega|-1)!$, and Cameron and Ku \cite{ref:CameronKu03} proved that equality holds if and only if $S$ consists of all permutations sending $\omega \in \Omega$ to $\omega' \in \Omega$. In particular, all extremal examples of intersecting \emph{subgroups} are point stabilisers, so have an orbit of size $1$. Do all intersecting subgroups of permutations have a small orbit, or perhaps even an orbit of size at most an absolute constant $C$? Nakajima \cite{ref:Nakajima} asked this with a view towards an application to constraint satisfaction problems if the question had an affirmative answer. A subgroup $G \leq \Sym(\Omega)$ is intersecting if and only if it has no derangements, so Remark~\ref{rem:coverings}(ii) negatively answers this question, but we conjecture that there are no examples whose smallest orbit has size $\frac{1}{2}|\Omega|$.
\end{shremark}

\begin{shremark} \label{rem:polynomials}
If $f \in \mathbb{Z}[X]$ has a root in $\mathbb{Z}$, then $f$ has a root modulo every positive integer~$m$, but the converse is false \cite{ref:BerendBilu96}. Let $f \in \mathbb{Z}[X]$ be a monic polynomial with no roots in $\mathbb{Z}$, and write $f = f_1 \cdots f_k$ where $f_i \in \mathbb{Q}[X]$ is irreducible with root $\alpha_i \in \overline{\mathbb{Q}}$. Then $f$ has a root modulo all but finitely many primes if and only if $\mathrm{Gal}(L/\mathbb{Q}) = \bigcup_{1 \leq i \leq k} \mathrm{Gal}(L/\mathbb{Q}(\alpha_i))$ where $L$ is the splitting field of $f$ \cite[Proposition~2.2]{ref:ElsholtzKlahnTechnau22}. Hence, Jordan's theorem implies that $f$ has no roots modulo infinitely many primes if $f$ is irreducible (see also \cite[Section~4]{ref:Serre03}). Moreover, Conjecture~\ref{conj:main} asserts that this also holds when $f = f_1f_2$ where $f_1$ and $f_2$ have the same degree, and Corollary~\ref{cor:main} proves this when $\mathbb{Q}(\alpha_1)/\mathbb{Q}$ has no proper intermediate subfields.
\end{shremark}

\begin{shremark} \label{rem:reduction}
Let us comment on the hypotheses in our main results.
\begin{enumerate}
\item Conjecture~\ref{conj:main} has no immediate reduction to primitive actions, as two subgroups of equal order need not be contained in maximal subgroups of equal order (for example, the subgroups $S_2 \times S_3$ and $A_4$ of $S_5$). This explains the additional hypothesis in Corollary~\ref{cor:main}.\label{rem:reduction_warning}
\item The hypothesis that $G$ acts primitively on $\Omega_2$ is necessary in Theorem~\ref{thm:main}, as a finite group can be the union of conjugates of proper subgroups $H_1$ and $H_2$ such that $|H_2|$ divides $|H_1|$. Indeed, Example~\ref{ex:example} gives a group $G = 2^3.A_4$ that is the union of conjugates of a maximal subgroup $H_1 \cong 2^3.C_3$ of order $24$ and a subgroup $H_2 \cong C_2 \times C_4$ of order $8$. \label{rem:reduction_example}
\end{enumerate}
\end{shremark}

Let us now outline our proofs. We first reduce Theorem~\ref{thm:main} to the case where $G$ acts faithfully and primitively on $\Omega_1$, and, using the O'Nan--Scott Theorem, we divide our analysis according to the possibilities for $G \leq \Sym(\Omega_1)$. However, even when $G \leq \Sym(\Omega_1)$ is of a fixed O'Nan--Scott type, the action of $G$ on $\Omega_2$ is an arbitrary (not necessarily faithful) primitive action, and our arguments are novel since we must keep track of how the actions of $G$ on $\Omega_1$ and $\Omega_2$ interact. We then reduce further to the case where $G \leq \Sym(\Omega_1)$ is almost simple or affine. For almost simple groups, we make use of recent work by Bubboloni, Spiga and Weigel \cite{ref:BubboloniSpigaWeigel} on normal $2$-coverings.

Affine groups require the most work, and we need a very recent representation theoretic result of Harper and Liebeck \cite{ref:HarperLiebeck25}, which generalises a result of Feit and Tits \cite{ref:FeitTits78}. (Indeed this paper was the original motivation for the work in \cite{ref:HarperLiebeck25}.) In order to apply this result, we need to carry out a series of nontrivial reductions. This is technically challenging since each of these reductions are in tension with each other and therefore need to be carried out in a certain order and often several times (see Section~\ref{ss:isbell_reduction}).

In addressing the affine groups, we also establish results which may be of independent interest. First, we obtain bounds on the smallest degrees of permutation and linear representations of simple groups $G$ in terms of the prime factorisation of $|G|$, see Section~\ref{ss:prelims_bounds}. Second, we establish a linear variant  of Isbell's Conjecture, as we now explain.

To answer a number theoretic question regarding relative Brauer groups, Fein, Kantor and Schacher \cite{ref:FeinKantorSchacher81} extended Jordan's theorem by proving that every nontrivial finite transitive permutation group has a derangement whose order is a power of a prime $p$. In contrast to Jordan's original theorem, which is an elementary counting argument, this generalisation uses the CFSG. Which prime $p$ works? Isbell's conjecture asserts that if $|\Omega| = p^ab$ for $a$ sufficiently large compared to $b$, then $G$ has a derangement whose order is a power of $p$. Motivated by a connection to fair $n$-player games, Isbell \cite{ref:Isbell60} made the conjecture for $p=2$, and the general conjecture was formulated by Cameron, Frankl and Kantor \cite[Section~1]{ref:CameronFranklKantor89}. We prove the following linear variant of Isbell's Conjecture.

\begin{shtheorem} \label{thm:isbell}
Let $G$ be a finite group acting primitively on $\Omega$. Let $p$ be a prime number and let $\rho\: G \to \GL_d(p)$ be a faithful irreducible representation. Assume that $p^d$ divides $|\Omega|$. Then $G$ contains a derangement $g$ such that $g\rho$ fixes a nonzero vector in $\F_p^d$.
\end{shtheorem}

\begin{shremark} \label{rem:isbell}
To see the connection to Isbell's Conjecture, note that if $g$ is a $p$-element, then $g\rho$ fixes a nonzero vector of $\F_p^d$. However, there is not always a derangement that is a $p$-element. For example, let $d = 2m \geq 6$, let $G = \GL_m(p^2)$, let $H = \GL_m(p)$ and consider the field extension embedding $\rho\: G \to \GL_d(p)$. Then $p^d$ divides $|G:H|$, but if $g \in G$ has $p$-power order, then $g$ is unipotent, so (via Jordan normal form) is conjugate to an element of $H$. Therefore, no $p$-element is a derangement in the primitive action of $G$ on $G/H$. 
\end{shremark}

Theorems~\ref{thm:summary} and~\ref{thm:reduction} are proved together in Section~\ref{s:reduction} with Theorem~\ref{thm:summary}(iii) involving computation with the perfect groups of order at most 50000 in \textsc{Magma} \cite{ref:Magma}.

\begin{notation*}
Our notation for the finite simple groups follows \cite{ref:KleidmanLiebeck}. In particular, we write $\PSL^+_n(q) = \PSL_n(q)$ and $\PSL^-_n(q) = \PSU_n(q)$ and also $E_6^+(q) = E_6(q)$ and $E_6^-(q) = {}^2E_6(q)$. We write $(a,b)$ for the greatest common divisor of positive integers $a$ and $b$.
\end{notation*}

\begin{acknowledgements*}
We thank Marston Conder for his contribution to Remark~\ref{rem:graphs}, Peter M\"uller for drawing our attention to \cite{ref:Klingen98}, Pablo Spiga for generously sharing his proof of Corollary~\ref{cor:soluble} and Gabriel Verret who, in response to an earlier version of this paper, pointed us to \cite{ref:GiudiciPotocnikVerret14}, which we now mention in Remark~\ref{rem:graphs}, and also to Feldman's StackExchange question \cite{ref:Feldman14}. We also thank the anonymous referee. The second author is an EPSRC Postdoctoral Fellow (EP/X011879/2). In order to meet institutional and research funder open access requirements, any accepted manuscript arising shall be open access under a Creative Commons Attribution (CC BY) reuse licence with zero embargo. No data was produced.
\end{acknowledgements*}

\section{Preliminaries} \label{s:prelims}

\subsection{Properties of primitive groups} \label{ss:prelims_primitive}

We begin with some preliminaries on primitive permutation groups. The O'Nan--Scott Theorem provides a case division for the primitive permutation groups, and the version we use is the main theorem of \cite{ref:LiebeckPraegerSaxl88} but with the labels used by Praeger \cite{ref:Praeger97}. The cases arising in the O'Nan--Scott Theorem are summarised in Table~\ref{tab:o'nan-scott}. We record some of the key properties we require in the following remark, and we refer the reader to \cite{ref:LiebeckPraegerSaxl88} for further information. 

\begin{table}[t]
\centering
\begin{tabular}{ll}
\hline
type & description \\
\hline
(HA) & affine: $G = p^k {:} H \leq \AGL_k(p)$ with $H \leq \GL_k(p)$ irreducible \\
(AS) & almost simple: $T \leq G \leq \Aut(T)$ \\
(SD) & diagonal-type: $T^k \leq G \leq T^k.(\Out(T) \times P)$ with $P \leq S_k$ primitive \\
(HS) & diagonal-type: $T^2 \leq G \leq T^2.\Out(T)$ \\
(PA) & product-type: $G \leq H \wr P$ with $H$ of type (AS) and $P \leq S_l$ transitive \\
(CD) & product-type: $G \leq H \wr P$ with $H$ of type (SD) and $P \leq S_l$ transitive \\
(HC) & product-type: $G \leq H \wr P$ with $H$ of type (HS) and $P \leq S_l$ transitive \\
(TW) & twisted wreath action \\
\hline
\end{tabular}
\caption{The primitive permutation groups (here $T$ denotes a nonabelian simple group).} \label{tab:o'nan-scott}
\end{table}

\begin{remark} \label{rem:o'nan-scott}
Recall that the \emph{socle} of a finite group $G$, denoted $\soc(G)$, is the product of the minimal normal subgroups of $G$. Let $G$ be a nontrivial finite primitive permutation group on a set $\Omega$. Then $\soc(G) = T^k$ where $T$ is a simple group and $k$ is a positive integer.
\begin{enumerate}
\item First assume that $T$ is abelian. This gives case (HA). In this case, $G$ is an \emph{affine group}, i.e. $G = V{:}H \leq \AGL(V)$ where $V$ is an elementary abelian group of order $p^k$ and $H \leq \GL(V)$ is an irreducible linear group. The action is given by the natural affine action of $\AGL(V)$ on $V$, viewed as the vector space $\F_p^k$, and $H$ is the stabiliser of the zero vector. Here $V$ is the unique minimal normal subgroup of $G$ and $V$ is regular. \label{rem:o'nan-scott_affine}
\item Next assume that $T$ is nonabelian and $k = 1$. This gives case (AS). In this case, $G$ is an \emph{almost simple group}, i.e. $T \leq G \leq \Aut(T)$, where $T$ is a nonabelian simple group. Here, $T$ is the unique minimal normal subgroup of $G$. \label{rem:o'nan-scott_as}
\item In the remaining cases, $T$ is nonabelian and $k \geq 2$. Therefore, $T^k \leq G \leq \Aut(T^k) = \Aut(T) \wr S_k$. A full description of the groups is given in \cite{ref:LiebeckPraegerSaxl88}. We will just restrict ourselves to discussing the minimal normal subgroups in this case. For (SD), (PA), (CD) and (TW), $\soc(G)$ is the unique minimal normal subgroup, and $\soc(G)$ is regular if and only if $G$ has type (TW). For (HS) and (HC), $G$ has exactly two minimal normal subgroups and these are isomorphic and regular. (For (HS), we always have $k=2$.) \label{rem:o'nan-scott_rest}
\item Parts~(i)--(iii) imply that if $G$ has no regular normal subgroups, then $\soc(G)$ is the unique minimal normal subgroup of $G$ and $\soc(G) \cong T^k$ where $T$ is nonabelian. \label{rem:o'nan-scott_regular}
\item From the description of the groups in \cite{ref:LiebeckPraegerSaxl88}, we see that the degree $|\Omega|$ is given as follows
\[
\begin{array}{ccccccccc}
\hline
\text{type} & \mathrm{(HA)} & \mathrm{(AS)} & \mathrm{(SD)} & \mathrm{(HS)} & \mathrm{(PA)}  & \mathrm{(CD)} & \mathrm{(HC)} & \mathrm{(TW)} \\
|\Omega|    & |T|^k         & |T:T_\omega|  & |T|^{k-1}     & |T|           & |T:T_\omega|^k & |T|^{k-l}     & |T|^{k/2}     & |T|^k         \\
\hline
\end{array}
\]
where $|T|$ is prime if and only if $G$ has type (HA), $l$ is a proper divisor of $k$ if $G$ has type (CD) and $\omega \in \Omega$ if $G$ has type (AS) or (PA). \label{rem:o'nan-scott_degree}
\end{enumerate}
\end{remark}

\subsection{An example} \label{ss:prelims_example}

In this short section, we give the details of the example referred to in Remark~\ref{rem:reduction}\ref{rem:reduction_example}.

\begin{example} \label{ex:example}
Let $G$ be the group with presentation
\[
\< x, y, z, t \mid x^4 = y^4 = z^2 = t^3 = [x,z] = [y,z] = 1, \ [x,y] = z, \ x^t = y, \ y^t = (xy)^{-1} \>,
\]
and consider the subgroups $H_1 = \< x^2, y^2, z, t\>$ and $H_2 = \< x, y^2 \>$. Then $G$ has order $96$, $H_1$ is a maximal subgroup of order $24$, $H_2$ is a subgroup of order $8$ and $G$ is the union of conjugates of $H_1$ and $H_2$. These claims are easily verified in \textsc{Magma}, but we will sketch the key ideas.

Let $N = \< x^2, y^2, z \> \cong 2^3$. Note that $N$ is normal in $G$ and $G/N \cong A_4$. Moreover, $N \leq H_1$ and $H_1/N = C_3$, which is a Sylow $3$-subgroup of $G/N$. In particular, $H_1$ is a maximal subgroup of $G$ and $|H_1| = 24$. Note that $H_2 = \<x\> \times \<y^2\> \cong C_4 \times C_2$, so $|H_2| = 8$.

We will now explain why $G$ is the union of conjugates of $H_1$ and $H_2$. Let $g \in G$. We claim that a suitable conjugate of $g$ is contained in $H_1$ or $H_2$. If $g \in N$, then certainly $g \in H_1$. If $Ng \in G/N$ has order $3$, then a $G/N$-conjugate of $Ng$ is contained in $H_1/N$, so a conjugate of $g$ is contained in $H_1$. It remains to assume that $Ng \in G/N$ has order $2$. By conjugating in $G/N \cong A_4$, we may assume that $Ng = Nx$. Said otherwise, $g = x^\pm y^{2i} z^j$ for some $i,j \in \{0,1\}$. Since $x^y = xz$, by conjugating $g$ by $y$ if necessary, we may assume that $g = x^\pm y^{2i}$, which is an element of $H_2$. This proves the claim.
\end{example}

\subsection{Derangements and primes} \label{ss:prelims_derangements}

In this section, we make some preliminary observations about derangements. We begin with a small generalisation of Jordan's original result, which was first noted in \cite[Section~3]{ref:Arvind17}.

\begin{lemma} \label{lem:jordan_coset}
Let $\Omega$ be finite, let $G \leq \Sym(\Omega)$ be transitive and let $h \in \Sym(\Omega)$. Then the average number of fixed points of elements in the coset $Gh$ is $1$.
\end{lemma}

\begin{proof}
The key observation is that for $g \in G$ and $\omega \in \Omega$, we have $\omega gh = \omega  \iff g \in G_\omega h^{-1}$. Now observe that
\[
\sum_{g \in G} \fix(gh) = |\{ (g,\omega ) \in G \times \Omega \mid \omega gh = \omega  \}| = \sum_{\omega  \in \Omega} |G_\omega | = |G|
\]
using the fact that $|G_\omega | = |G|/|\Omega|$ since $G$ is transitive.
\end{proof}

\begin{remark} \label{rem:jordan_coset}
Consider Lemma~\ref{lem:jordan_coset} when $h$ is trivial. The average number of fixed points of $G \leq \Sym(\Omega)$ is $1$. If $|\Omega| > 1$, then the identity element has more than one fixed point, so $G$ must contain an element with no fixed points, which returns Jordan's original theorem.
\end{remark}

For the next lemma, we use a famed theorem \cite{ref:FeinKantorSchacher81}, which depends on the CFSG.

\begin{theorem*}[Fein, Kantor \& Schacher, 1981]
Let $\Omega$ be finite. Let $1 < G \leq \Sym(\Omega)$ be transitive. Then $G$ contains a derangement of prime-power order.
\end{theorem*}

\begin{lemma} \label{lem:technical}
Let $G \leq \Sym(\Omega)$ have exactly two orbits $\Omega_1$ and $\Omega_2$. Assume that $|\Omega_1|, |\Omega_2| > 1$ and that no prime divisor of $|\Omega_1|$ divides $|\Omega_2|-1$. Let $N_1$ be the kernel of the action of $G$ on $\Omega_1$. Assume that $N_1$ acts transitively on $\Omega_2$. Then $G$ has a derangement.
\end{lemma}

\begin{proof}
Viewing $G$ as a subgroup of $\Sym(\Omega_1) \times \Sym(\Omega_2)$, for all $g \in G$, write $g = (g_1,g_2)$ where $g_1 \in \Sym(\Omega_1)$ and $g_2 \in \Sym(\Omega_2)$. Since $N_1$ acts transitively on $\Omega_2$, by Lemma~\ref{lem:jordan_coset}, the average number of fixed points on $\Omega_2$ of elements in any coset of $N_1$ in $G$ is $1$.  Suppose that for every $g \in G$ such that $g_1 \in \Sym(\Omega_1)$ is a derangement, every element of the coset $N_1g$ has exactly one fixed point on $\Omega_2$. By \cite[Theorem~1]{ref:FeinKantorSchacher81}, there exists an element $g \in G$ such that $g_1 \in \Sym(\Omega_1)$ is a derangement of order a power of some prime $p$. In particular, $|\Omega_1|$ is divisible by $p$. Consider the corresponding permutation $g_2 \in \Sym(\Omega_2)$. By replacing $g$ by a suitable $p'$-power if necessary, we may assume that every cycle of $g_2$ has length a power of $p$ (while maintaining the condition that $g_1$ is a derangement). However, $g_2$ has exactly one fixed point, so $|\Omega_2|-1$ is divisible by $p$, which contradicts our hypothesis. Therefore, there must exist $g \in G$ such that $g_1 \in \Sym(\Omega_1)$ is a derangement and $h \in N_1g$ such that $h_2 \in \Sym(\Omega_2)$ is a derangement. However, since $N_1$ acts trivially on $\Omega_1$, we know that $h_1 = g_1 \in \Sym(\Omega_1)$ is also a derangement, so $h$ is a derangement on $\Omega$.
\end{proof}

\subsection{Bounds for simple groups} \label{ss:prelims_bounds}

We will now establish some new bounds on invariants associated with finite simple groups that we will use at various points in the proofs that follow. To state these bounds, we need some notation. Let $p$ be a prime number, let $n$ be a positive integer, let $G$ be a finite group and define the following (where $H \preccurlyeq G$ means that $H$ is isomorphic to a subgroup of $G$):
\begin{align*}
v_p(n) &= \max\{ d \mid \text{$p^d$ divides $n$} \}. \\
P(G)   &= \min\{ d \mid G \preccurlyeq S_d \} \\[2pt]
R_p(G) &= \min\{ d \mid G \preccurlyeq \PGL_d(\overline{\F}_p) \} \\[2pt]
n_G'   &= \min\{ n \mid G \preccurlyeq \GL_{2n}(2) \text{ irreducible} \}.
\end{align*}

The main result of this section is the following.

\begin{proposition} \label{prop:p-part}
Let $G$ be a nonabelian finite simple group and let $p$ be prime. Then
\begin{enumerate}[{\rm (i)}]
\item $v_p(|G|) \leq P(G)$
\item $v_p(|G|) \leq 2^{n_G'}$ if $p \neq 2$
\item $v_p(|G|) \leq R_p(G)$ if neither of the following hold
\begin{enumerate}[{\rm (a)}]
\item $G$ is a finite simple group of Lie type in characteristic $p$
\item $p=2$ and $G \in \mathcal{E}$ where $\mathcal{E} = \{ A_8, \, \PSU_4(3), \, {\rm M}_{22}, \, {\rm J}_2, \, {\rm Suz} \}$.
\end{enumerate}
\end{enumerate}
\end{proposition}

A \emph{splitting field} for a finite group $G$ is a field $F$ that is minimal subject to the property that every irreducible representation over $\overline{F}$ is expressible over $F$. Defining characteristic splitting fields for the finite simple groups of Lie type are given in \cite[Proposition~5.4.4]{ref:KleidmanLiebeck}.

\begin{lemma} \label{lem:fields}
Let $G$ be a nonabelian finite simple group of Lie type defined over $\F_{2^f}$, and let $\F_{2^{fu}}$ be a splitting field for $G$. Then
\[
R_2(G) \cdot f \leq 2n_G' \leq R_2(G) \cdot fu.
\]
\end{lemma}

\begin{proof}
By definition, there is a faithful projective representation $G \to \PGL_{R_2(G)}(2^{fu})$, which is necessarily irreducible since $G$ is nonabelian simple. Composing with the field extension embedding $\PGL_{R_2(G)}(2^{fu}) \to \PGL_{R_2(G) \cdot fu}(2)$, we obtain a faithful irreducible representation $G \to \PGL_{R_2(G) \cdot fu}(2) = \GL_{R_2(G) \cdot fu}(2)$, which proves that $2n_G' \leq R_2(G) \cdot fu$.

By definition, there is a faithful irreducible representation $G \to \GL_{2n_G'}(2) = \PGL_{2n_G'}(2)$. Tensoring with $\End_{\F_2G}(V) = \F_{2^e}$, we obtain a faithful absolutely irreducible representation $G \to \PGL_{2n_G'/e}(2^e)$ (see \cite[Lemma~2.10.2]{ref:KleidmanLiebeck}). Now we choose $d \leq e$ to be minimal such that this representation is expressible over $\F_{2^d}$, thus yielding a faithful absolutely irreducible representation $G \to \PGL_{2n_G'/e}(2^d)$. Applying \cite[Proposition~5.4.6 \& Remark~5.4.7]{ref:KleidmanLiebeck}, which are consequences of Steinberg's twisted tensor product theorem, we deduce that 
\[
2n_G'/e \geq R_2(G)^{f/d} \geq R_2(G) \cdot f/d \geq R_2(G) \cdot f/e,
\] 
which proves that $2n_G' \geq R_2(G) \cdot f$, as claimed. 
\end{proof}

\begin{lemma} \label{lem:bound} \quad
\begin{enumerate}
\item Let $d \geq 4$ and $r,p \geq 2$ with $r \neq p$ and $(d,r) \not\in \{ (4,2), (4,3), (5,2) \}$. Then 
\[
(4r+4)^d \leq p^{r^{d-2}-1}.
\]
\item Let $b,p \geq 3$ and $f \geq 1$. Then 
\[
(b+1)^2 f \leq p^{bf}-1.
\]
\end{enumerate}
\end{lemma}

\begin{proof}
This is easily verified.
\end{proof}

\begin{proof}[Proof of Proposition~\ref{prop:p-part}]
We divide into cases according to $G$.

\emph{\textbf{Case~1.} $G$ is an alternating group.}\nopagebreak

Let $G = A_m$. Recall that $v_p(m!) = (m-s_p(m))/(p-1)$ where $s_p(m)$ is the sum of the digits in the base $p$ expansion of $m$. In particular, $v_p(|G|) \leq v_p(m!) \leq m$. Part~(i) follows since $P(G) = m$. Now consider part~(ii), so $p >2$. If $m > 8$, then $2n_{G}' \geq R_2(G) \geq m-2$ (see \cite[Proposition~5.3.7]{ref:KleidmanLiebeck}(i)), so 
\[
v_p(|G|) \leq m/2 \leq 2^{(m-2)/2} \leq 2^{R_2(|G|)/2} \leq 2^{n_G'}.
\]
For $m \leq 8$, it suffices to note that $v_p(|G|) \leq 2$ since $p$ is odd. Finally consider part~(iii). For now assume that $m > 8$, so $R_p(G) \geq m-2$ (see \cite[Proposition~5.3.7]{ref:KleidmanLiebeck}(i)). We claim that $v_p(|G|) \leq m-2$. If $p > 2$, then $v_p(|G|) \leq m/2 \leq m-2$, and if $p=2$, then
\[
v_2(|G|) = v_2(m!/2) = v_2(m!)-1 = m-s_2(m)-1 \leq m-2.
\]
For $m \leq 8$, consulting \cite[Proposition~5.3.7]{ref:KleidmanLiebeck}(ii), if $v_p(|G|) > R_p(G)$ then $m=8$ and $p=2$.

\emph{\textbf{Case~2.} $G$ is sporadic.}\nopagebreak

In this case, $v_p(|G|)$ can be read off from the factorised order of $G$ given in \cite[Table~5.1.C]{ref:KleidmanLiebeck} and $R(G) = \min \{ R_p(G) \mid \text{$p$ is prime}\}$ is given in \cite[Proposition~5.3.8]{ref:KleidmanLiebeck}, and from these it is easy to observe that if $v_p(|G|) > R(G)$, then $p=2$ and $G \in \mathcal{E}$. Part~(iii) follows immediately from this observation. The observation also gives part~(i) since $P(G) \geq R(G)$ and the three exceptions in $\mathcal{E}$ can be verified using the $\mathbb{ATLAS}$ \cite{ref:ATLAS}. Finally, the observation gives part~(ii) since $R(G) \leq 2n_{G}' \leq 2^{n_{G}'}$ and the three exceptions only arise when $p=2$.

\emph{\textbf{Case~3.} $G = {}^2F_4(2)'$.}\nopagebreak

Here $|G| = 2^{11} \cdot 3^3 \cdot 5^2 \cdot 13$. Using \textsc{Magma} we see that $P(G) = 1600$ and consulting \cite[Chapter~5]{ref:KleidmanLiebeck} we see that $R_p(G) \geq 26$, so the result holds.

\emph{\textbf{Case~4.} $G$ is a group of Lie type over $\F_r$.}\nopagebreak

We exclude $\PSL_2(4)$, $\PSL_2(5)$ and $\PSL_4(2)$ on account of their isomorphisms with alternating groups. It is easy to verify the result in \textsc{Magma} for the following groups
\begin{equation} \label{eq:p-part_computation}
\begin{array}{c}
 F_4(2), \ G_2(4), \ \PSL_2(r) \, (r \leq 16), \ \PSL^\pm_3(4), \ \PSL^\pm_4(2),    \ \PSL^\pm_4(3), \ \PSL^\pm_5(2),               \\[2pt]
\PSp_6(2),     \ \PSp_8(2),        \ \PSp_8(3),     \ \PSp_{10}(2), \ \Omega_9(3), 
\POm^\pm_{8}(2), \ \POm^\pm_8(3), \ \POm^\pm_{10}(2), 
\end{array}
\end{equation}
so we will exclude these groups from the analysis that follows too.

\begin{table}[b]
\[
\begin{array}{ccccccccccccccc}
\hline
   & A^\pm_{\ell \geq 2} & B_{\ell \geq 2}  & C_{\ell \geq 3} & D_{\ell \geq 4} & E_8 & E_7 & E_6^\pm & F_4 & G_2 & A_1 & {}^3D_4 & {}^2F_4 & {}^2B_2 & {}^2G_2 \\
\hline
e  & (\ell^2+\ell)/2     & \ell^2           & \ell^2          & \ell^2-\ell     & 120 & 63  & 36      & 24  & 6   & 1   & 12      & 12      & 2       & 3       \\
d  & \ell+1              & \ell             & \ell            & \ell            & 15  & 9   & 12      & 6   & 3   & 1   & 6       & 6       & 2       & 3       \\
c  & \ell+1              & 2\ell+1          & 2\ell           & 2\ell           & 248 & 56  & 27      & 25  & 6   & 2   & 8       & 26      & 4       & 7       \\
b  & \ell-1              & \ell             & \ell-1          & \ell            & 28  & 16  & 10      & 7   & 2   & 0   & 4       & 4       & 1       & 1       \\
\hline
\end{array}
\]
\caption{The parameters $a$, $b$, $c$ and $d$.} \label{tab:p-part_abcd}
\end{table}

Define $e$, $d$, $c$, $b$ in Table~\ref{tab:p-part_abcd}. Note that $b \geq d-2$. By consulting the order formula in \cite[Tables~5.1.A \& 5.1.B]{ref:KleidmanLiebeck} we see that the $r$-part of $|G|$ is $r^e$ and $r'$-part of $|G|$ divides $\prod_{i=1}^{d}(a^i-1)$ for some $a \in \{r,-r,r^2\}$, so, by \cite[p.~464]{ref:Artin55}, we deduce that if $p$ does not divide $r$, then $p^{v_p(|G|)} \leq (4r+4)^d$. In addition, from \cite[Table~5.4.C]{ref:KleidmanLiebeck}, if $p$ divides $r$, then $R_p(G) \geq c$. If $p$ does not divide $r$, then it is easy to deduce from \cite[Table~5.3.A]{ref:KleidmanLiebeck} that $R_p(G) \geq r^b-1$, where we use the fact that $G$ is not in \eqref{eq:p-part_computation}.

We now prove each part of the statement, but not in order.

\emph{\textbf{Part~(iii).}} Here we may assume that $p \neq r$. First assume that $d \geq 4$. Note that $(d,r) \not\in \{ (4,2), (4,3), (5,2) \}$ since $G$ is not in \eqref{eq:p-part_computation}. Then, by Lemma~\ref{lem:bound}(i), 
\[
v_p(|G|) \leq r^{d-2}-1 \leq r^b-1 \leq R_p(G).
\]
Now assume that $d \leq 3$. By \cite[Theorem~5.3.9]{ref:KleidmanLiebeck} (and since $G$ is not in \eqref{eq:p-part_computation}), $R_p(G) \geq a(r)$ where $a(r)$ is given in Table~\ref{tab:p-part_small}. With this it is easy to verify that 
\[
v_p(|G|) \leq d \cdot \log_p (4r+4) \leq a(r) \leq R_p(G).
\]

\begin{table}[b]
\begin{align*}
& \begin{array}{ccccccc}
\hline
G    & \PSL_2(r)               & \PSL_3(r) & \PSL_4(r) & \PSU_3(r) & \PSU_4(r)   & \PSp_4(r)          \\ 
a(r) & \frac{1}{(2,r-1)} (r-1) & r^2-1     & r^3-1     & r^2-r     & r^3-r^2+r-1 & \frac{1}{2}(r^2-1) \\
\hline
\end{array} \\
& \hspace{12mm} \begin{array}{ccccc}
\hline
\PSp_6(r)          & \Omega_7(r) & G_2(r) & {}^2B_2(r)              & {}^2G_2(r) \\
\frac{1}{2}(r^3-1) & r^4-1       & r^3-r  & \sqrt{\frac{r}{2}}(r-1) & r^2-r      \\
\hline
\end{array}
\end{align*}
\caption{The function $a(r)$ for groups $G$ with $b \leq 2$ or $d \leq 3$.} \label{tab:p-part_small}
\end{table}

\emph{\textbf{Part~(i).}} First assume that $p \neq r$. Then $v_p(|G|) \leq R_p(G) \leq P(G)$. Now assume that $p=r$. Then $|G|_p = r^e$ and we check that $v_p(|G|) \leq P(G)$ by consulting the values of $P(G)$ given in \cite[Theorem~5.2.2]{ref:KleidmanLiebeck} if $G$ is classical and the main theorems of \cite{ref:Vasilyev96,ref:Vasilyev97,ref:Vasilyev98} if $G$ is exceptional.

\emph{\textbf{Part~(ii).}} First assume that neither $p$ nor $2$ divides $r$. As in part~(iii), if $d \geq 4$, then
\[
v_p(|G|) \leq r^{d-2}-1 \leq r^b-1 \leq R_2(G) \leq 2n_G' \leq 2^{n_G'},
\] 
and if $d \leq 3$, then
\[
v_p(|G|) \leq d \cdot \log_p (4r+4) \leq a(r) \leq R_2(G).
\]

Next assume that $p$ divides $r$ (so $2$ does not divide $r$), and write $r = p^f$. If $b \geq 3$, then, noting that $e \leq (b+1)^2$, by Lemma~\ref{lem:bound}(ii), 
\[
v_p(|G|) = ef \leq (b+1)^2f \leq p^{bf}-1 = r^b-1 \leq R_2(G) \leq 2n_G' \leq 2^{n_G'}.
\]
Now assume that $b \leq 2$. From the information in Table~\ref{tab:p-part_small}, it is easy to check that $a(r) \geq e$, recalling that $G$ is not in \eqref{eq:p-part_computation}. Therefore, by Lemma~\ref{lem:fields}, 
\[
v_p(|G|) = ef \leq a(r)f \leq R_2(G) f \leq 2n_G' \leq 2^{n_G'}.
\]

Finally assume that $2$ divides $r$ (so $p$ does not divide $r$), and write $r=2^f$. Then, noting that $d \leq c$, we have
\[
v_p(|G|) \leq d \cdot \log_p (4 \cdot 2^f + 4) \leq 2^{df/2} \leq 2^{cf/2}
\]
unless perhaps $f=1$ and $d \leq 8$ (noting that $f \geq 5$ when $d = 1$ since $G$ is not in \eqref{eq:p-part_computation}). However, in these cases, by considering the possibilities for $|G|$ (in \textsc{Magma} \cite{ref:Magma}, say) we see that we still have $v_p(|G|) \leq 2^{c/2} = 2^{cf/2}$. Hence, by Lemma~\ref{lem:fields}, 
\[
v_p(|G|) \leq 2^{cf/2} \leq 2^{R_2(G)f/2} \leq 2^{n_G'}. \qedhere
\]
\end{proof}

\section{Linear variant on Isbell's Conjecture} \label{s:isbell}

The purpose of this section is to prove Theorem~\ref{thm:isbell}. In Section~\ref{ss:isbell_simple} we prove Proposition~\ref{prop:isbell_simple}, which is an important special case of Theorem~\ref{thm:isbell} that serves as the base case for our proof. In Section~\ref{ss:isbell_reduction} we prove Lemma~\ref{lem:isbell_reduction}, which is a technical result which will play a key role in reducing Theorem~\ref{thm:isbell} to the special case handled in Proposition~\ref{prop:isbell_simple}. In Section~\ref{ss:isbell_proof}, we complete the proof of Theorem~\ref{thm:isbell}, where the four main tools are the base case in Proposition~\ref{prop:isbell_simple}, the bounds in Proposition~\ref{prop:p-part}, the reduction in Lemma~\ref{lem:isbell_reduction} and \cite[Theorem~4]{ref:HarperLiebeck25}. The last of these is a recent representation theoretic result established by Harper and Liebeck.

\subsection{Base case for Theorem~\ref{thm:isbell}} \label{ss:isbell_simple}

This section is dedicated to proving a special case of Theorem~\ref{thm:isbell}. We first give two lemmas. In the first, we label the simple roots of indecomposable root systems according to the convention of Bourbaki \cite{ref:Bourbaki68}, which is consistent with \cite{ref:KleidmanLiebeck} (see, in particular, \cite[(5.2.2)]{ref:KleidmanLiebeck}).

\begin{lemma} \label{lem:e_roots}
Let $\ell \in \{6,7,8\}$, let $\Phi$ be the $E_\ell$ root system and let $\{ \alpha_1, \dots, \alpha_\ell \}$ be the set of simple roots of $E_\ell$. Let $\alpha  = \sum_{1 \leq i \leq \ell} m_i \alpha_i \in \Phi$. If $m_\ell > 0$ and $m_{\ell-1} = 0$, then $\alpha = \alpha_\ell$.
\end{lemma}

\begin{proof}
Following \cite[Planches~V--VII]{ref:Bourbaki68}, fix the standard basis $e_1, \dots, e_8$ for $\mathbb{R}^8$ and write
\begin{gather*}
\alpha_1 = \tfrac{1}{2}(e_1+e_8) - \tfrac{1}{2}(e_2+e_3+e_4+e_5+e_6+e_7), \quad \alpha_2 = e_1+e_2, \\
\alpha_3 = e_2 - e_1, \quad \alpha_4 = e_3 - e_2, \quad \alpha_5 = e_4 - e_3, \\
\alpha_6 = e_5 - e_4, \quad \alpha_7 = e_6 - e_5, \quad \alpha_8 = e_7 - e_6.
\end{gather*}
Then $\alpha = (a_1,\dots,a_8)$ is
\begin{gather*}
(\tfrac{1}{2}m_1 + m_2 - m_3, \  -\tfrac{1}{2}m_1 + m_2+m_3-m_4, \ -\tfrac{1}{2}m_1 + m_4-m_5, \ -\tfrac{1}{2}m_1 + m_5-m_6, \\
-\tfrac{1}{2}m_1 + m_6 - m_7, \ -\tfrac{1}{2}m_1 + m_7 - m_8, \ -\tfrac{1}{2}m_1 + m_8, \ \tfrac{1}{2}m_1).
\end{gather*}

First assume that $\ell = 8$, so $m_8 > 0$ and $m_7 = 0$. 
If all coordinates of $\alpha$ are $\pm \frac{1}{2}$, 
then $a_7 = -\frac{1}{2}m_1+m_8 = \frac{1}{2}$, so $m_1 = m_8 = 1$, but then $a_6 = -\frac{3}{2}$, a contradiction. 
Therefore, $\alpha$ is a permutation of $(1,-1,0,0,0,0,0,0)$. 
In particular, $m_8 = 1$, $a_7 = -\frac{1}{2}m_1+1 \in \{1,0,-1\}$ and $a_6 = -\frac{1}{2}m_1-1 \in \{1,0,-1\}$, 
so $m_1 = 0$ and $(a_6,a_7,a_8) = (-1,1,0)$. Therefore, $\alpha = \alpha_6$.

Next assume that $\ell = 7$, so $m_8 = m_6 = 0$ and $m_7 > 0$. 
If all coordinates of $\alpha$ are $\pm \frac{1}{2}$, 
then $a_6 = -\frac{1}{2}m_1+m_7 = \frac{1}{2}$, so $m_1 = m_7 = 1$, but then $a_5 = - \frac{3}{2}$, a contradiction. 
Therefore, $\alpha$ is a permutation of $(1,-1,0,0,0,0,0,0)$. 
In particular, $m_7 = 1$, $a_6 = -\frac{1}{2}m_1+1 \in \{1,0,-1\}$ and $a_5 = -\frac{1}{2}m_1-1 \in \{1,0,-1\}$, 
so $m_1 = 0$ and $(a_5,a_6,a_7,a_8) = (-1,1,0,0)$. Therefore, $\alpha = \alpha_7$.

Finally assume that $\ell = 6$, so $m_8 = m_7 = m_5 = 0$ and $m_6 > 0$. 
If all coordinates of $\alpha$ are $\pm \frac{1}{2}$,
then $a_5 = -\frac{1}{2}m_1+m_6 = \frac{1}{2}$, so $m_1 = m_6 = 1$, but then $a_4 = - \frac{3}{2}$, a contradiction. 
Therefore, $\alpha$ is a permutation of $(1,-1,0,0,0,0,0,0)$. 
In particular, $m_6 = 1$ and $a_5 = -\frac{1}{2}m_1+1 \in \{1,0,-1\}$ and $a_4 = -\frac{1}{2}m_1-1 \in \{1,0,-1\}$, 
so $m_1 = 0$ and $(a_4,a_5,a_6,a_7,a_8) = (-1,1,0,0,0)$. Therefore, $\alpha = \alpha_6$.
\end{proof}

For the second lemma, we write $C_V(g)$ for the $1$-eigenspace of $g \in \GL(V)$.

\begin{lemma} \label{lem:fixed_space}
Let $p$ be prime and let $V = \F_q^n$ where $q = p^f$. Let $g \in \GL_n(q)$ and assume that $|g|$ divides $q^e - 1$ but is coprime to $q^i-1$ for all $1 \leq i < e$. Then $\dim C_V(g) \equiv \dim V \mod e$.
\end{lemma}

\begin{proof}
Since $g$ is semisimple, $\< g \>$ stabilises a direct sum decomposition $V = \bigoplus_{1 \leq i \leq k} V_i$, acting irreducibly on $V_i$ for each $1 \leq i \leq k$. For each $1 \leq i \leq k$, let $g_i \in \GL(V_i)$ be the restriction of $g$ to $V_i$. Reordering the summands if necessary, we may fix $1 \leq \ell \leq k$ such that $g_i = 1$ if and only if $i > \ell$. This means that $\dim C_V(g) = \dim V - \sum_{1 \leq i \leq \ell} \dim V_i$.

Let $1 \leq i \leq \ell$, so $|g_i| > 1$. Since $g_i$ is irreducible, $|g_i|$ divides $q^{\dim V_i}-1$, but $|g_i|$ also divides $|g|$, which divides $q^e-1$, so $|g_i|$ divides $q^{(\dim V_i, e)}-1$. However, $|g|$ is coprime to $q^i-1$ for all $1 \leq i < e$, which implies that $(\dim V_i,e) = e$, or said otherwise, $e$ divides $\dim V_i$. Therefore, $e$ divides $\sum_{1 \leq i \leq \ell} \dim V_i$, which proves that $\dim C_V(g) \equiv \dim V \mod e$. 
\end{proof}

\begin{proposition} \label{prop:isbell_simple}
Let $p$ be prime, let $T$ be a finite simple group of Lie type defined over $\F_q$ where $q=p^f$ and let $\rho\: T \to \PGL_n(\F_{p^a})$ be a faithful absolutely irreducible projective representation. Let $T \leq A \leq \Aut(T)$, let $\widetilde{M}$ be a maximal subgroup of $A$ not containing $T$ and let $M = \widetilde{M} \cap T$. Assume that $p$ divides $|M|$ and $p^{an}$ divides $|T:M|$. Then there exists $g \in T$ such that $g$ is a derangement in the action of $T$ on $T/M$ and $g\rho$ fixes a nonzero vector in $\F_{pa}^n$.
\end{proposition}

\begin{proof}
Fix $a_0 \leq a$ such that $\rho$ is expressible over no proper subfield of $\F_{p^{a_0}}$ and consider the corresponding projective representation $\rho_1\: T \to \PGL_n(p^{a_0})$. Let $G$ be a perfect central extension of $T$ such that $\rho_1$ lifts to a faithful representation $\rho_2\: G \to \GL_n(p^{a_0})$. Let $S$ be the full covering group of $T$ and consider the corresponding representation $\lambda\: S \to \GL_n(p^{a_0})$. Note that $\lambda$ is an absolutely irreducible representation expressible over no proper subfield. Let $V = \F_{p^{a_0}}^n$ be the module afforded by $\lambda$, and let $U = \F_{q^u}^m$ be the minimal module for $S$, so $u \in \{1,2,3\}$ depending on $T$.

First assume that $T$ is very twisted (that is, $T \in \{ {}^2B_2(q), {}^2G_2(q), {}^2F_4(q) \}$), so $f$ is odd. Here, by \cite[Remark~5.4.7(ii)]{ref:KleidmanLiebeck}, $n \geq m^{f/a_0}$, but this means that
\[
fm \leq a \cdot f/a_0 \cdot m \leq a m^{f/a_0} \leq  an < v_p(|T|) = fe
\]
where $e$ is given in Table~\ref{tab:p-part_abcd}. Noting $m \geq c$ in Table~\ref{tab:p-part_abcd}, we obtain a contradiction. 

Now assume that $T$ is not very twisted. If $T$ is untwisted, then, by \cite[Proposition~5.4.6(i)]{ref:KleidmanLiebeck}, $a_0$ divides $f$ and there exists an absolutely irreducible $\F_{p^f}S$-module $W = \F_{p^f}^t$ such that $n = t^{f/a_0}$ and 
\[
V \otimes \F_{p^f} = W \otimes W^{(a_0)} \otimes \cdots \otimes W^{(f-a_0)}.
\]

If $T$ has type ${}^2A_\ell$, ${}^2D_\ell$ or ${}^2E_6$, then, by \cite[Proposition~5.4.6(ii)]{ref:KleidmanLiebeck} $a_0$ divides $2f$ and there exists an absolutely irreducible module $W = \F_{p^f}^t$ such that either $W$ is stable under the graph automorphism, $n = t^{f/a_0}$ and 
\[
V \otimes \F_{p^f} = W \otimes W^{(a_0)} \otimes \cdots \otimes W^{(f-a_0)},
\] 
or $W$ is not stable under the graph automorphism, $n = t^{2f/a_0}$ and 
\[
V \otimes \F_{p^{fu}} = W \otimes W^{(a_0)} \otimes \cdots \otimes W^{(2f-a_0)}.
\] 
In both cases, $n \geq t^{f/a_0}$.

If $T$ has type ${}^3D_4$, then, by \cite[Remark~5.4.7(i)]{ref:KleidmanLiebeck} there exists an absolutely irreducible module $W = \F_{p^f}^t$ such that either $W$ is stable under the graph automorphism and $n = t^{f/a_0}$, $W$ is not stable under the graph automorphism and $n = t^{3f/a_0}$.

In all cases,
\[
ft \leq a \cdot f/a_0 \cdot t \leq e t^{f/a_0} \leq an  < v_p(|T|) = fe
\]
where, again, $e$ is given in Table~\ref{tab:p-part_abcd} and $m \geq c$ in Table~\ref{tab:p-part_abcd}. By \cite[Propositions~5.4.11 and~5.4.12]{ref:KleidmanLiebeck}, the condition $c \leq m \leq t < e$ forces one of the following
\begin{enumerate}
\item $T$ is classical, ${}^3D_4(q)$, $E_6^\pm(q)$ or $E_7(q)$ and $W$ is quasiequivalent to $U$
\item $T$ is $B_3(q)$, $C_3(q)$ or $D^\pm_5(q)$ and $W$ is quasiequivalent to the spin module (of dimension 8, 8, 16, respectively).
\end{enumerate}

If $T = {}^2E_6(q)$, then $t < e = 36$ implies that $W$ is not stable under the graph automorphism (see \cite[Proposition~5.4.8]{ref:KleidmanLiebeck}), so the argument above implies that we actually have $2m \leq 2t < e$, which gives a contradiction, so the case $T = {}^2E_6(q)$ does not arise. Similarly, if $T = {}^3D_4(q)$, then $t < e = 8$ again implies that $W$ is not stable under the graph automorphism, so $3m \leq 3t < e$, which is another contradiction, so $T = {}^3D_4(q)$ does not occur either.

We now divide the remaining possibilities into several cases depending on $T$.

\emph{\textbf{Case~1.} $T \in \{ \PSL_2(q), \PSL_3(q), \PSU_3(q), \PSp_4(q) \}$.}\nopagebreak

If $T = \PSL_2(q)$, then the $p$-part of $T$ is $q$, so there are no subgroups $M$ of $T$ such that $q^{\dim V} = q^2$ divides $|T:M|$ and the result follows vacuously. Otherwise, the $p$-part of $|T|$ is $q^{\dim V}$, so $p$ does not divide $|M|$, which means that any element $g \in T$ of order $p$ is a derangement in the action of $T$ on $T/M$ and fixes a nonzero vector of $V$ since $g$ is unipotent.

\emph{\textbf{Case~2.} $T$ is a classical group not in Case~1.}\nopagebreak

We will begin by excluding three small cases. If $T = \PSU_4(3)$ and $M = 2^4.A_6$, then any element of order $7$ satisfies the theorem. If $T = \PSU_6(2)$ and $\soc(M) = \PSU_4(3)$, then any element of order $11$ satisfies the theorem. If $T = \Omega_7(5)$ and $M = 2^6{:}A_7$, then any element of order $25$ satisfies the theorem. For the rest of the proof, we will assume that $(T,M)$ is not one of these possibilities.

Since any element of order $p$ fixes a nonzero vector of $V$, it suffices to assume that $T$ has no derangement of order $p$ in the action on $T/M$. By \cite[Theorems~5.3.1, 5.4.1, 5.6.1, 5.7.1 \& 5.9.1]{ref:BurnessGiudici16}, $M$ is not in the Aschbacher class $\C_3$, $\C_4$, $\C_6$, $\C_7$ or $\N$, and by \cite[Theorem~1]{ref:BurnessGiudici18}, $M$ is not in $\S$ (here we are making use of the fact that we excluded the first two small cases in the previous paragraph). Therefore, one of the following holds:
\begin{enumerate}[(a)]
\item $M$ is reducible, or $q$ is even, $T = \Sp_n(q)$ and $M = \mathrm{O}^\pm_n(q)$ 
\item $M$ is irreducible but imprimitive
\item $M$ is expressible over a proper subfield of $\F_{q^u}$
\item $M$ is a classical group with natural module $V$ (here $T = \PSL_n(q)$).
\end{enumerate}

For now assume that $V$ is the natural module for $T$. 

Let $\iota = 1$ if $T \in \mathcal{A}_1$ and let $\iota = 2$ if $T \in \mathcal{A}_2$, where
\begin{align*}
\mathcal{A}_1 &= \{ \PSL_n(q), \ \PSU_n(q) \, \text{($n$ even)}, \ \Omega_n(q) \, \text{($n$ odd)} \} \\
\mathcal{A}_2 &= \{ \PSU_n(q) \, \text{($n$ odd)}, \ \PSp_n(q) \, \text{($n$ even)}, \ \POm^\pm_n(q) \, \text{($n$ even)} \}.
\end{align*}
Fix a direct sum decomposition $V = V_1 \oplus V_2$ with $\dim V_2 = \iota$. Let $g = g_1 \oplus g_2 \in T$ be an element that stabilises this decomposition of maximal possible order such that $g_1$ acts irreducibly on $V_1$ and $g_2$ acts trivially on $V_2$. Notice that such elements do exist by the choice of $\iota$. Note that $g$ fixes a nonzero vector of $V$. We consider the cases (a)--(d) above.

First suppose that (a) holds. Since $g$ acts irreducibly on a nondegenerate $\ell$-space, one of the following holds:
\begin{enumerate}
\item $T \in \{ \PSL_n(q), \PSU_n(q), \Omega_n(q) \}$ and $M$ is the stabiliser of a nondegenerate $1$-space
\item $T \in \{ \PSU_n(q), \PSp_n(q), \POm^\pm_n(q) \}$ and $M$ is the stabiliser of a totally singular $1$-space or a nondegenerate $2$-space
\item $T  = \Sp_n(q)$ and $M = \mathrm{O}^\pm_n(q)$.
\end{enumerate}
It suffices to prove that $q^{un}$ does not divide $|T:M|$. If $M$ is a parabolic subgroup, then $M$ contains a Sylow $p$-subgroup of $T$, so $|T:M|$ is not divisible by $p$ and the result follows. If $M$ is not a parabolic subgroup, then the result can be readily checked via the order formulae in \cite[Table~5.1.A]{ref:KleidmanLiebeck}. For instance, in (iii), writing $n=2m$, the $p$-part of $|T|$ is $q^{m^2}$ and the $p$-part of $|M|$ is $q^{m^2-m}$, so the $p$-part of $|T:M|$ is $q^m$, which is not divisible by $q^n$. 

Next suppose that (b) holds. Consulting the main theorem of \cite{ref:GuralnickPenttilaPraegerSaxl97} for example (specifically Example~2.3), we see that $\ell+1$ is the only primitive prime divisor of $(q^u)^\ell-1$, $M$ is the stabiliser of a decomposition of $V$ into $1$-spaces and $g$ induces an $(\ell+1)$-cycle on these $1$-spaces. However, the order of $g$ is impossibly large for this to hold. For example, if $T = \PSU_n(q)$ and $n$ is odd, then $|g| = (q^{n-2}+1)/(q+1)$, which strictly exceeds $(n-1) (q+1)$, which is at least the maximal element order of an element of $M$. Therefore, (b) does not hold.

Now suppose that (c) holds. The order of $g$ is divisible by a primitive prime divisor of $q^{u\ell}-1$. In all cases, $\ell > n/2$, so $M$ does not contain any elements whose order is divisible by a primitive prime divisor of $(q^u)^\ell-1$. Therefore, (c) does not hold.

Now suppose that (d) holds. Here $T = \PSL_n(q)$ and $g$ has order $\frac{q^{n-1}-1}{q-1}$, which exceeds the order of any element of $M$ that has order divisible by a primitive prime divisor of $q^{n-1}-1$. Therefore, (d) does not hold.

It remains to assume that $V$ is the spin module for $T$. 

First assume that $T$ is $\Sp_6(q)$ with $q$ even or $\Omega_7(q)$ with $q$ odd. Let $r$ be the greatest divisor of $q^6-1$ that is coprime to $q^i-1$ for $1 \leq i < 6$, and let $g \in T$ have order $r$, which exists since $r$ divides $q^3+1$. (If $q=2$, then $r=1$, but this will not cause problems.) Since $\dim V = 8 \equiv 2 \mod 6$, Lemma~\ref{lem:fixed_space} implies that $g$ fixes a nonzero vector of $V$. It remains to prove that $g$ is a derangement on $T/M$. The $p$-part of $|T|$ is $q^9$ and $q^{\dim V} = q^8$ divides $|T:M|$, so the $p$-part of $|M|$ is at most $q$. Consulting the list of maximal subgroups of $T$ in \cite[Tables~8.28 \&~8.39]{ref:BrayHoltRoneyDougal} and bearing in mind that we are assuming that $|M|$ satisfies one of (a)--(d) above, we see that the only possibilities are that $M$ is a subfield subgroup defined over $\F_{q^{1/k}}$ where $k \geq 5$ (so, in particular, $q \geq 2^5$) or $T = \Omega_7(q)$ with $q > 3$ and $M$ is an imprimitive subgroup of type $\mathrm{O}_1(q) \wr S_7$. The former case can be eliminated on the grounds that $|M|$ is not divisible by a primitive prime divisor of $q^6-1$ (which exists since $q > 2$), so assume that $T = \Omega_7(q)$ and $M$ is an imprimitive subgroup of type $\mathrm{O}_1(q) \wr S_7$. In light of the opening paragraph of Case~2, we may assume that $q \neq 5$. Note that $r \equiv 1 \mod 6$ and, since $q \not\in \{3,5\}$, by \cite[Theorem~3.9]{ref:Hering74}, $r \geq 13$. As such $S_7$ contains no elements of order $r$, which excludes this possibility for $M$. 

Now assume that $T$ is $\POm^\pm_{10}(q)$. If $\e = +$, then let $r$ be the greatest divisor of $q^5-1$ that is coprime to $q^i-1$ for $1 \leq i < 5$, and if $\e = -$, then let $r$ be the greatest divisor of $q^{10}-1$ that is coprime to $q^i-1$ for $1 \leq i < 10$. Since $\dim V = 16 \equiv 1 \mod 5$, Lemma~\ref{lem:fixed_space} implies that $g$ fixes a nonzero vector of $V$. It remains to prove that $g$ is a derangement on $T/M$. The $p$-part of $|T|$ is $q^{20}$ and $q^{\dim V} = q^{16}$ divides $|T:M|$, so the $p$-part of $|M|$ is at most $q^4$.Consulting \cite[Tables~8.66 \&~8.68]{ref:BrayHoltRoneyDougal}, the only possibilities are that $M$ is a subfield subgroup defined over $\F_{q^{1/k}}$ where $k \geq 5$ or $M$ is an imprimitive subgroup of type $\mathrm{O}_1(q) \wr S_{10}$ or $\mathrm{O}^\pm_2(q) \wr S_5$. In the former case, $|M|$ is not divisible by a primitive prime divisor of $q^5-1$ or $q^{10}-1$, so assume that $M$ is imprimitive. If $\e = +$, then $r \equiv 1 \mod 5$ and  \cite[Theorem~3.9]{ref:Hering74} implies that $r \geq 11$, and if $\e = -$, then $r \equiv 1 \mod {10}$, so again, $r \geq 11$. Therefore, in both cases, $M$ is excluded since $S_{10}$ contains no elements of order $r$.

\emph{\textbf{Case~3.} $T$ is $E_6(q)$ or $E_7(q)$ and $V$ is the minimal module.}\nopagebreak

We follow the description of the minimal module for $T$ given in \cite[p.203]{ref:KleidmanLiebeck}.  We divide into two similar cases. 

\emph{\textbf{Case~3a.} $T = E_6(q)$.}\nopagebreak

Let $d = (3,q-1)$ and let $\widetilde{T} = d.T$ be the full covering group of $T$. Let $P$ be the $P_7$ parabolic subgroup of $E_7(q)$, with Levi decomposition $P = QL$. Then $Q$ is an elementary abelian group of order $q^{27}$ and $L = KH$ where $K \cong \widetilde{T}$ and $H$ is a Cartan subgroup of $E_7(q)$ that normalises $Q$. Identifying $V$ with $Q$ and $\widetilde{T}$ with $K$, the action of $T$ on $V$ is obtained via the conjugation action of $\widetilde{T}$ on $Q$. Let $P_0$ be the $P_6$ parabolic of $E_6(q)$, with Levi decomposition $P_0 = Q_0L_0$. Then $L_0 = K_0H_0$ where $K_0 \cong \Spin^+_{10}(q)$ and $H_0$ is a Cartan subgroup of $E_6(q)$ that normalises $Q_0$. Let $\widetilde{K}_0$ subgroup of $\widetilde{T}$ corresponding to $K_0$.

Let $\Phi$ be the $E_7$ root system with simple roots $\{ \alpha_1, \dots, \alpha_7 \}$. Then 
\[
Q = \< X_\alpha \mid \text{$\alpha = \sum_{1 \leq i \leq 7} m_i \alpha_i \in \Phi$ with $m_i \geq 0$ and $m_7 > 0$} \>.
\]
Let $t$ be an nonzero element of $\F_q$, and under the identification of $V$ with $Q$, let $v \in V$ be the nonzero vector corresponding to the nontrivial element $x_{\alpha_7}(t)$ of the root subgroup $X_{\alpha_7}$. We claim that every element of $K_0$ fixes $v$. Said otherwise, we claim that every element of $\widetilde{K}_0$ commutes with $x_{\alpha_7}$. Let $x_\alpha(s) \in \widetilde{K}_0$ where $s \in \F_q$ and $\alpha \in \Phi$. We may write $\alpha = \sum_{1 \leq i \leq 5} m_i \alpha_i$ since, by construction,
\[
\widetilde{K}_0 = \< X_{\alpha} \mid \text{$\alpha = \sum_{1 \leq i \leq 7} m_i \alpha_i \in \Phi$ with $m_6 = m_7 = 0$} \>.
\]
Hence, Lemma~\ref{lem:e_roots} implies that $j\alpha + k\alpha_7 \not\in \Phi$ for all $j,k>0$, so, by the Chevalley commutator formula,
\[
[x_\alpha(s),x_{\alpha_7}(t)] = \prod_{j,k > 0} x_{j\alpha+k\alpha_7}(C^{j,k}_{\alpha,\alpha_7}s^jt^k) = 1,
\]
which proves the claim (here $C^{j,k}_{\beta,\gamma}$ are the structure constants).

Therefore, it suffices to identify an element $g \in K_0$ that is a derangement on $T/M$. Let $g \in K_0 \cong \Spin^+_{10}(q)$ have order $(q^4+1)(q+1)$. The possibilities for $M$ are given \cite[Tables~2 \&~9]{ref:Craven23}. The $p$-part of $|T|$ is $q^{36}$ and $q^{27}$ divides $|T:M|$, so the $p$-part of $|M|$ is at most $q^9$. In particular, $M$ is not a parabolic subgroup, a subfield subgroup defined over $\F_{q^{1/2}}$ or a subgroup of type $C_4(q)$, $F_4(q)$ or $D_5(q) \times (q-1)$. Since $|g|$ is at least $51$ and divisible by a primitive prime divisor of $q^8-1$, we quickly deduce that $g$ is not contained in any of the remaining possibilities for $M$.

\emph{\textbf{Case~3b.} $T = E_7(q)$.}\nopagebreak

Let $d = (2,q-1)$ and let $\widetilde{T} = d.T$ be the full covering group of $T$. Let $P$ be the $P_8$ parabolic subgroup of $E_8(q)$, with Levi decomposition $P = QL$. Then $Q \cong q^{1+56}$ and $L = KH$ where $K \cong \widetilde{T}$ and $H$ is a Cartan subgroup of $E_8(q)$ that normalises $Q$. Identifying $V$ with $Q/Z(Q) \cong q^{56}$ and $\widetilde{T}$ with $K$, the action of $T$ on $V$ is obtained via the conjugation action of $\widetilde{T}$ on $Q$. Let $P_0$ be the $P_7$ parabolic of $E_7(q)$, with Levi decomposition $P_0 = Q_0L_0$. Then $L_0 = K_0H_0$ where $K_0 \cong (3,q-1).E_6(q)$ and $H_0$ is a Cartan subgroup of $E_7q)$ that normalises $Q_0$. Let $\widetilde{K}_0$ subgroup of $\widetilde{T}$ corresponding to $K_0$.

Let $\Phi$ be the $E_8$ root system with simple roots $\{ \alpha_1, \dots, \alpha_8 \}$. Then 
\[
Q = \< X_\alpha \mid \text{$\alpha = \sum_{1 \leq i \leq 8} m_i \alpha_i \in \Phi$ with $m_i \geq 0$ and $m_8 > 0$} \>.
\]
The Chevalley commutator formula implies that $Z(Q) = X_{\alpha_0}$, where $\alpha_0$ is the highest root in $\Phi$. Let $t$ be an nonzero element of $\F_q$, and under the identification of $V$ with $Q/Z(Q)$, let $v \in V$ be the nonzero vector corresponding to the coset of $X_{\alpha_0}$ with representative $x_{\alpha_8}(t)$. We claim that every element of $K_0$ fixes $v$. Said otherwise, we claim that every element of $\widetilde{K}_0$ commutes with $x_{\alpha_8}$ modulo $X_{\alpha_0}$. Let $x_\alpha(s) \in \widetilde{K}_0$ where $s \in \F_q$ and $\alpha \in \Phi$. We may write $\alpha = \sum_{1 \leq i \leq 6} m_i \alpha_i$ since, by construction,
\[
\widetilde{K}_0 = \< X_{\alpha} \mid \text{$\alpha = \sum_{1 \leq i \leq 8} m_i \alpha_i \in \Phi$ with $m_7 = m_8 = 0$} \>.
\]
Hence, Lemma~\ref{lem:e_roots} implies that $j\alpha + k\alpha_8 \not\in \Phi$ for all $j,k>0$, so
\[
[x_\alpha(s),x_{\alpha_7}(t)] = \prod_{j,k > 0} x_{j\alpha+k\alpha_8}(C^{j,k}_{\alpha,\alpha_8}s^jt^k) = 1,
\]
which proves the claim.

Therefore, it suffices to identify an element $g \in K_0$ that is a derangement on $T/M$. Let $g \in K_0 \cong (3,q-1).E_6(q)$ have order $q^6+q^3+1$. The $p$-part of $|T|$ is $q^{63}$ and $q^{56}$ divides $|T:M|$, so the $p$-part of $|M|$ is at most $q^7$. In particular, $M$ is not a parabolic subgroup, a subgroup of type $E_6(q) \times (q-1)$ or a subfield subgroup defined over $\F_{q^{1/2}}$. Note that $|g|$ is at least $73$ and divisible by a primitive prime divisor of $q^9-1$. 

The maximal subgroups of almost simple groups with socle $E_7(q)$ have not yet been determined (except when $q=2$), but much is known and the existing literature is ample for our purposes. In particular, writing $T = O^{p'}(X_\sigma)$ where $X$ is a simple linear algebraic group of type $E_6$ of adjoint type and $\sigma$ is a Frobenius endomorphism, $M$ satisfies one of the following 
\begin{enumerate}[{\rm (I)}]
\item $Y_\sigma \cap T$ for a maximal closed $\sigma$-stable positive-dimensional subgroup $Y$ of $X$
\item $X_\alpha \cap T$ for a Steinberg endomorphism $\alpha$ of $X$ such that $\alpha^k=\sigma$ for a prime $k$
\item a local subgroup not in (I)
\item an almost simple group not in (I) or (II).
\end{enumerate}

First assume that $M$ satisfies (I). Then $M$ is given in \cite[Tables~5.1 \&~5.2]{ref:LiebeckSaxlSeitz92} if $Y$ has maximal rank and \cite[Table~III]{ref:LiebeckSeitz90} otherwise. No such groups contain an element of order divisible by a primitive prime divisor of $q^9-1$ (recalling that $M$ is not a parabolic subgroup or a subgroup of type $E_6(q) \times (q-1)$). Next assume that $M$ satisfies (II), so $M = E_7(q^{1/k})$ for a prime $k > 2$ (recalling that $M$ is not a subfield subgroup defined over $\F_{q^{1/2}}$), but no such subgroup has order divisible by a primitive prime divisor of $q^9-1$. Now assume that $M$ satisfies (III). Then $M$ is given in \cite[Table~1]{ref:CohenLiebeckSaxlSeitz92} and again $M$ contains no elements of suitable order. For the rest of the proof, we may assume that $M$ satisfies (IV). If $\soc(M)$ is not a group of Lie type in characteristic $p$, then, as noted in \cite[Proposition~3.7]{ref:BurnessGuralnickHarper21}, no element of $M$ has order exceeding $63$, but $|g| \geq 73$, so $g$ is not contained in $M$. Hence, we may assume that $\soc(M)$ is a group of Lie type in characteristic $p$. By \cite[Theorem~1]{ref:LiebeckSeitz98TAMS}, $\soc(M)$ is one of the following
\begin{enumerate}
\item $\PSL_2(t)$, ${}^2B_2(t)$ or ${}^2G_2(t)$ with $t \leq (2,p-1) \cdot 388$
\item $\PSL^\pm_3(t)$ with $t \leq 16$
\item $\PSL^\pm_4(t)$, $\PSp_4(t)$, $\PSp_6(t)$, $\Omega_7(t)$ or $G_2(t)$ with $t \leq 9$
\end{enumerate}
Consulting the order formulae for these groups, if $|M|$ is divisible by a primitive prime divisor of $q^9-1$, then $q = p = 2$ and $\soc(M)$ is one of the following
\[
\PSL_3(2^3), \ \PSL_4(2^3), \ \PSp_6(2^3), \ G_2(2^3),
\]
but, in all cases, the $2$-part of $|M|$ exceeds $2^7$, which is a contradiction. Therefore, $g$ is not contained in $M$, which completes the proof.
\end{proof}

\subsection{Reduction lemma for Theorem~\ref{thm:isbell}} \label{ss:isbell_reduction}

For this section, fix a prime number $p$. We will establish some useful reductions we can make in our proof of Theorem~\ref{thm:isbell}. The following carefully stated hypothesis is a key aspect of this.

\begin{hypothesis} \label{hyp:isbell_reduction} 
Let
\begin{enumerate}
\item $T$ be a nonabelian finite simple group
\item $M$ be a subgroup of $T$ with order divisible by $p$
\item $k$ be a positive integer
\item $Q$ be $T^k$ or a monolithic group with monolith $T^k$ such that $\bigcup_{g \in Q}(M^k)^g = \bigcup_{g \in T^k}(M^k)^g$
\item $G$ be a finite group
\item $N$ be a normal subgroup of $G$ such that $G/N = Q$
\item $\gamma$ be the quotient map $\gamma\: G \to Q$
\item $d$ be a positive integer
\item $\rho$ be a faithful irreducible representation $G \to \GL_d(p)$
\item $H$ be a subgroup of $G$ such that $N \leq H$ and $H\gamma \cap T^k = M^k$
\end{enumerate}
Assume that $p^d$ divides $|T:M|^k$. Then there exists $g \in G$ that is a derangement on $G/H$ and such that $g\rho$ fixes a nonzero vector of $\F_p^d$.
\end{hypothesis} 

In the context of Hypothesis~\ref{hyp:isbell_reduction}, we write $T^k = T_1 \times \cdots \times T_k$ and $M^k = M_1 \times \cdots \times M_k$. We will also write $V$ for the $\F_pG$-module afforded by $\rho$.

The next lemma shows that if there is a counterexample to Hypothesis~\ref{hyp:isbell_reduction}, then there is a counterexample with several convenient properties. We write $\Frat(G)$ for the Frattini subgroup of a group $G$.

\begin{lemma} \label{lem:isbell_reduction}
Assume that Hypothesis~\ref{hyp:isbell_reduction} has a counterexample with the nonabelian simple group $T$. Then Hypothesis~\ref{hyp:isbell_reduction} has a counterexample $(T,M,k,Q,G,N,\gamma,d,\rho,H)$ where the following hold
\begin{enumerate}
\item $Q = T^k$
\item $N \leq \Frat(G)$
\item $\rho$ is primitive.
\end{enumerate}
\end{lemma}

We begin with some preliminary reductions. For the first, recall that for a finite group $G$ and a normal subgroup $N \leqn G$, we have $N \leq \Frat(G)$ if and only if $KN < G$ for all $K < G$.

\begin{lemma} \label{lem:isbell_minimal}
Let $C = (T,M,k,Q,G,N,\gamma,d,\rho,H)$ be a counterexample to Hypothesis~\ref{hyp:isbell_reduction}. Then there exists a counterexample $C_0 = (T,M,k,Q,G_0,N_0,\gamma_0,d_0,\rho_0,H_0)$ where 
\begin{enumerate}
\item $d_0 \leq d$
\item $N_0 \leq \Frat(G_0)$
\end{enumerate} 
\end{lemma}

\begin{proof}
Let $G_1 \leq G$ be minimal such that $G_1N = G$. Let $V_0$ be an irreducible $\F_pG_1$-submodule of $V|_{G_1}$ of dimension $d_0$, and let $\rho_1\: G_1 \to \GL(V_0)$ be $g \mapsto (g\rho)|_{V_0}$. 

We claim that $\ker\rho_1 \leq N$. Suppose otherwise. Since $(\ker\rho_1)\gamma$ is a nontrivial normal subgroup of $G_1\gamma = Q$, without loss of generality, $T_1 \leq (\ker\rho_1)\gamma$. Let $g_1 \in T_1$ such that $g_1$ is contained in no $\Aut(T_1)$-conjugate of $M_1$ (which is possible by \cite[Proposition~2]{ref:Saxl88}), and let $g \in \ker\rho_1$ such that $g\gamma = g_1$. Then $g\gamma$ is a derangement in the action of $G\gamma$ on $G\gamma/H\gamma$, and hence $g$ is a derangement in the action of $G$ on $G/H$. However, $g$ fixes every vector in $V_0$, which contradicts $C$ being a counterexample. Therefore, $\ker\rho_1 \leq N$, as claimed.

Write $G_0 = G_1\rho_1$, so the inclusion map $\rho_0\:G_0 \to \GL(V_0)$ is a faithful irreducible representation. Write $N_1 = G_1 \cap N$ and $N_0 = N_1\rho_1$. Since $\ker\rho_1 \leq N_1$ we have 
\[
G_0/N_0 \cong G_1/N_1 \cong G_1N/N = G/N = Q.
\] 
Write $\gamma_0\: G_0 \to Q$ for the quotient map defined by $N_0 \leqn G_0$, and note that $\gamma|_{G_1} = \rho_1\gamma_0$. Fix $N_0 \leq H_0 \leq G_0$ such that $H_0/N_0 = H/N$. Then $p^{d_0}$ divides $p^d$ which divides $|T:M|^k$.

Suppose there exists $g_0 \in G_0$ that is a derangement on $G_0/H_0$ and such that $g_0\rho_0$ fixes a nonzero vector of $V_0$. Fix $g \in G_1 \leq G$ such that $g\rho_1 = g_0$. Then $g\rho$, which agrees with $g_0\rho_0$ on $V_0$, fixes a nonzero vector of $V_0 \leq V$ and $g$ is a derangement in the action of $G$ on $G/H$ since $g\gamma = g_0\gamma_0$, which contradicts $C$ being a counterexample. Therefore, $(T,M,k,Q,G_0,N_0,\gamma_0,d_0,\rho_0,H_0)$ is a counterexample. 

We claim that $N_0 \leq \Frat(G_0)$. Let $K_0 < G_0$. Let $\ker\rho_1 \leq K_1 < G_1$ such that $K_1\rho_1 = K_0$. If $K_1N_1 = G_1$, then $K_1N = G_1N = G$, but this contradicts the choice of $G_1$, so $K_1N_1 < G_1$. Therefore, $K_0N_0 = (K_1\rho_1)(N_1\rho_1) < G_1\rho_1 = G_0$, since $\ker\rho_1 \leq K_1$ and $\ker\rho_1 \leq N_1$. This proves that $N_0 \leq \Frat(G_0)$, as claimed. Therefore, $C_0$ satisfies the lemma.
\end{proof}

For the next reduction, we need the following technical lemma, which is \cite[Corollary~2.6]{ref:HarperLiebeck25}. Recall that a \emph{subdirect product} of $G_1 \times \cdots \times G_r$ is a subgroup $G \leq G_1 \times \cdots \times G_r$ such that for all $1 \leq i \leq r$ the projection $G \to G_i$ defined as $(g_1, \dots, g_r) \mapsto g_i$ is surjective.

\begin{lemma} \label{lem:subdirect}
Let $G$ be a subdirect product of $G_1 \times \cdots \times G_r$ and let $\pi_i\: G \to G_i$ be the projection onto the $i$th factor. Assume that $N$ is a soluble normal subgroup of $G$ such that $G/N \cong T^k$ where $T$ is a nonabelian simple group and $k$ is a positive integer. Then there exist nonnegative integers $k_1, \dots, k_r$ satisfying $k_1 + \cdots + k_r \geq k$ such that $G_i/N\pi_i \cong T^{k_i}$ for all $1 \leq i \leq r$.
\end{lemma}

\begin{lemma} \label{lem:isbell_simple}
Let $C = (T,M,k,Q,G,N,\gamma,d,\rho,H)$ be a counterexample to Hypothesis~\ref{hyp:isbell_reduction} where $N$ is soluble. Then there exists a counterexample $C_0 = (T,M,k_0,Q_0,G_0,N_0,\gamma_0,d_0,\rho_0,H_0)$ where 
\begin{enumerate}
\item $d_0 \leq d$
\item $Q_0 = T^{k_0}$
\end{enumerate} 
\end{lemma}

\begin{proof}
Fix $N \leq G_1 \leq G$ such that $G_1\gamma = T^k$. Since $\rho$ is irreducible and $G_1 \leqn G$, by Clifford's theorem, $V|_{G_1} = V_1 \oplus \cdots \oplus V_r$ where $V_i$ is irreducible of dimension $d_0 = d/r$. For each $i$, let $\pi_i\: G_1 \to \GL(V_i)$ be $g \mapsto g|_{V_i}$, so $G_1$ is a subdirect product of $G_1\pi_1 \times \cdots \times G_1\pi_r$. Since $N$ is soluble, by Lemma~\ref{lem:subdirect}, we can fix $1 \leq i \leq r$ such that $G_1\pi_i/N\pi_i \cong T^{k_0}$ for $k_0 \geq k/r$. 

Write $G_0 = G_1\pi_i$ and $V_0 = V_i$, so the inclusion map $\rho_0\: G_0 \to \GL(V_0)$ is a faithful irreducible representation. Write $N_0 = N\pi_i$ and let $\gamma_0\: G_0 \to Q_0$ be the quotient map defined by $N_0 \leqn G_0$, where $Q_0 = T^{k_0}$. Fix $N_0 \leq H_0 \leq G_0$ such that $H_0/N_0 = M^k\pi_i \cong M^{k_0}$. Note that $p^{d_0} = p^{d/r}$ divides $|T:M|^{k_0}$ since $p^d$ divides $|T:M|^k$ and $k_0 \geq k/r$.

Suppose that there exists $g_0 \in G_0$ that is a derangement on $G_0/H_0$ and such that $g_0\rho_0$ fixes a nonzero vector of $V_0$. In particular, $g_0\gamma_0$ is contained in no $T^{k_0}$-conjugate of $M^{k_0}$. Fix $g \in G_1 \leq G$ such that $g\pi_i = g_0$. Then $g\rho$, which agrees with $g_0\rho_0$ on $V_0$, fixes a nonzero vector of $V_0 \leq V$, and $g\gamma$, which agrees with $g_0\gamma_0$ in the components corresponding to $T^{k_0}$, is contained in no $T^k$-conjugate of $M^k$. By part~(iv) of Hypothesis~\ref{hyp:isbell_reduction}, $g\gamma$ is contained in no $Q$-conjugate of $M^k$ either. Therefore, $g$ is contained in no $G$-conjugate of $H$, or said otherwise, $g$ is a derangement in the action of $G$ on $G/H$, which contradicts $C$ being a counterexample. Therefore, $(T,M,k_0,Q_0,G_0,N_0,\gamma_0,d_0,\rho_0,H_0)$ is a counterexample. 
\end{proof}

We can now prove Lemma~\ref{lem:isbell_reduction}. Here we make use of the bound in Proposition~\ref{prop:p-part}(i).

\begin{proof}[Proof of Lemma~\ref{lem:isbell_reduction}]
Let 
\[ 
C_1 = (T,M,k_1,Q_1,G_1,N_1,\gamma_1,d,\rho_1,H_1)
\]
be a counterexample to Hypothesis~\ref{hyp:isbell_reduction} with $d$ chosen minimally. By Lemma~\ref{lem:isbell_minimal}, there exists a counterexample 
\[
C_2 = (T,M,k_1,Q_1,G_2,N_2,\gamma_2,d,\rho_2,H_2)
\] 
where $N_2 \leq \Frat(G_2)$. In particular, $N_2$ is soluble (in fact, nilpotent), so, by Lemma~\ref{lem:isbell_simple}, there exists a counterexample 
\[
C_3 = (T,M,k_3,Q_3,G_3,N_3,\gamma_3,d,\rho_3,H_3)
\] 
where $G_3/N_3 = T^{k_3}$. Then, by Lemma~\ref{lem:isbell_minimal}, there exists a counterexample 
\[
C = (T,M,k,Q,G,N,\gamma,d,\rho,H)
\] 
where $G/N = T^k$ and $N \leq \Frat(G)$. 

It remains to prove that $\rho$ is primitive. Suppose otherwise. Let $V = V_1 \oplus \cdots \oplus V_r$ be a maximal system of linear imprimitivity. Let $\p\:G \to S_r$ be the permutation representation of $G$ on the summands, noting that $\p$ is transitive since $\rho$ is irreducible. Let $G_0$ be the stabiliser of $V_1$, so $|G:G_0| = r$, and let $K = \ker\p$, so, without loss of generality, $K\gamma = T_1 \times \cdots \times T_{\ell}$ for some $1 \leq \ell \leq k$. Hence, $K$ is the core of $G_0$ in $G$, so $K\gamma = T^\ell$ is the core of $G_0\gamma$ in $G\gamma = T^k$. Therefore, $G_0\gamma$ corresponds to a core-free subgroup of $G\gamma/K\gamma = T^{k-\ell}$. This means that
\[
d \geq r = |G:G_0| \geq |G\gamma:G_0\gamma| \geq P(T)^{k-\ell}.
\]
However, $p^d$ divides $|T:M|^k$ and $p$ divides $|M|$, so
\[
d \leq v_p(|T:M|) \cdot < v_p(|T|) \cdot k \leq P(T) \cdot k
\]
where the final inequality is given by Proposition~\ref{prop:p-part}(i). Hence, $k > P(T)^{k-\ell-1} \geq 5^{k-\ell-1}$, so $k \geq 2$ and $\ell > k - \log_5 k - 1$, which implies that $\ell \geq k/2$.

Let $\rho_0\:K \to \GL(V_1)$ be the restriction $g \mapsto g|_{V_1}$. We claim that $\ker\rho_0 \leq N$. Suppose otherwise. Since $(\ker\rho_0)\gamma$ is a nontrivial normal subgroup of $K\gamma = T^\ell$, without loss of generality, $T_1 \leq (\ker\rho_0)\gamma$. Fix $g \in \ker\rho_0$ such that $g\gamma \in T_1$ is contained in no $T_1$-conjugate of $M_1$. Then $g\gamma$ is contained in no $T^k$-conjugate of $M^k$, so $g$ is a derangement in the action of $G$ on $H$. However, $g$ fixes every vector in $V_1$, which contradicts $C$ being a counterexample. Therefore, $\ker\rho_0 \leq N$ as claimed.

Write $K_0 = K\rho_0$ and $N_0 = (N \cap K)\rho_0$. Since $\ker\rho_0 \leq N \cap K$ we have 
\[
K_0/N_0 \cong K/(N \cap K) \cong KN/N \cong K\gamma = T^\ell.
\] 
Write $\gamma_0\: K_0 \to T^\ell$ for the quotient map defined by $N_0 \leqn K_0$. Note that $\gamma|_K = \rho_0\gamma_0$. Fix $N_0 \leq H_0 \leq K_0$ such that $H_0/N_0 = M^\ell$. 

By construction, $G_0\rho_0 \leq \GL(V_1)$ is primitive and $K_0 \leqn G_0\rho_0$, so $V_1|_{K_0} = U_1 \oplus \cdots \oplus U_t$ where $U_i$ is an irreducible $\F_pK_0$-module of dimension $d^* = \dim V_1/t = d/rt$. For each $i$, let $\pi_i\: K_0 \to \GL(U_i)$ be $g \mapsto g|_{U_i}$, so $K_0$ is a subdirect product of $K_0\pi_1 \times \cdots \times K_0\pi_t$. Since $N_0$ is soluble (being a subquotient of the soluble group $N$), by Lemma~\ref{lem:subdirect}, we can fix $1 \leq i \leq t$ such that $K_0\pi_i/N_0\pi_i \cong T^{k^*}$ for $k^* \geq \ell/t$. Since $\ell \geq k/2 \geq k/r$, we have $k^* \geq k/rt$. 

Write $G^* = K_0\pi_i$, $N^* = N_0\pi_i$ and $Q^* = G^*/N^* = T^{k^*}$, and let $\gamma_0$ be the quotient map defined by $N^* \leqn G^*$. Write $V^* = U_i$, so the inclusion map $\rho^*\: G^* \to \GL(V^*)$ is a faithful irreducible representation. Fix $N^* \leq H^* \leq G^*$ such that $H^*/N^* = M^k\pi_i \cong M^{k^*}$. Note that $p^{d^*} = p^{d/rt}$ divides $|T:M|^{k^*}$ since $p^d$ divides $|T:M|^k$ and $k^* \geq k/rt$.

Suppose that there exists $g^* \in G^*$ that is a derangement on $G^*/H^*$ and such that $g^*\rho^*$ fixes a nonzero vector of $V^*$. In particular, $g^*\gamma^*$ is contained in no $T^{k^*}$-conjugate of $M^{k^*}$. Fix $g \in K \leq G$ such that $g\rho_0\pi_i = g^*$. Then $g\rho$, which agrees with $g^*\rho^*$ on $V^*$, fixes a nonzero vector of $V^* \leq V$, and $g\gamma$, with agrees with $g^*\gamma^*$ on $T^{k^*}$, is contained in no $T^k$-conjugate of $M^k$, so $g$ is a derangement in the action of $G$ on $G/H$, which contradicts $C$ being a counterexample. Therefore, $C^* = (T,M,k^*,Q^*,G^*,N^*,\gamma^*,d^*,\rho^*,H^*)$ is a counterexample

Now $d^* = d/rt < d$ since $r > 1$, which contradicts the minimality of $d$.  Therefore, $\rho$ is primitive, as required.
\end{proof}

\subsection{Proof of Theorem~\ref{thm:isbell}} \label{ss:isbell_proof}

To prove Theorem~\ref{thm:isbell}, we use the main results of the previous two sections and the bounds in Proposition~\ref{prop:p-part}(ii) and~(iii). We also need a recent result of Harper and Liebeck \cite[Theorem~4]{ref:HarperLiebeck25}, generalising a result of Feit and Tits \cite{ref:FeitTits78}. Recall the notation $n_G'$ from Section~\ref{ss:prelims_bounds}.

\begin{theorem*}[Harper \& Liebeck, 2024]
Let $\gamma\: G \to T^k$ be a surjective homomorphism where $G$ is a finite group, $T$ is a nonabelian simple group, $k$ is a positive integer and $H\gamma < T^k$ for all $H < G$. Let $F$ be an algebraically closed field, and let $\lambda\: G \to \PGL_n(F)$ be a faithful primitive projective representation. Assume that $n < 2^{n'_T} \cdot k$ if $\mathrm{char}\, F \neq 2$. Then $\gamma$ is an isomorphism.
\end{theorem*}

\begin{proof}[Proof of Theorem~\ref{thm:isbell}]
Let $N$ be the kernel of the action of $G$ on $\Omega$. First assume that the faithful primitive action of $G/N$ on $\Omega$ has type (AS) or (PA) (see Table~\ref{tab:o'nan-scott}). Since $p$ divides $|\Omega|$, \cite[Theorem~2.1]{ref:BurnessGiudiciWilson11} ensures that there exists an element $g \in G$ of order $p$ that is a derangement on $\Omega$, and since it has order $p$, the element $g\rho$ is a unipotent element of $\GL_d(p)$ and hence fixes a nonzero vector of $\F_p^d$, as required. 

For the remainder of the proof we may assume that the action of $G/N$ on $\Omega$ has type (AS) or (PA). Let $\gamma\: G \to Q$ be the quotient map defined by $N \leqn G$, write $\soc(G/N) = T^k$ where $T$ is simple and fix $N \leq G_0 \leq G$ such that $G_0\gamma = T^k$. There exists an almost simple group $A$ with socle $T$ and a maximal subgroup $\widetilde{M}$ of $A$ such that $H\gamma \cap T^k = M^k$ where $M = \widetilde{M} \cap T$. Since $G$ acts primitively on $G/H$ and $G_0$ is a normal subgroup of $G$ not contained in the kernel of the action, we deduce that $G = G_0H$. 

For a contradiction, suppose that there does not exist $g \in G$ such that $g$ is a derangement in the action of $G$ on $\Omega$ and such that $g\rho$ fixes a nonzero vector of $\F_p^d$. 

We claim that $C = (T,M,k,Q,G,N,\gamma,d,\rho,H)$ is a counterexample to Hypothesis~\ref{hyp:isbell_reduction}. There are a few points to check.  

First note that $p^d$ divides
\[
|G:H| = |G_0H:H| = |G_0:(H \cap G_0)| = |G_0\gamma:(H \cap G_0)\gamma| = |T^k:M^k| = |T:M|^k.
\]

Next we claim that $p$ divides $|M|$. Suppose otherwise. Let $g_1 \in T^k$ have order $p$, and let $g_2 \in G_0$ such that $g_2\gamma = g_1$. By raising $g_2$ to a suitable power coprime to $p$, we obtain a $p$-element $g \in G_0$ such that $g\gamma = g_1$. Since $g$ is a $p$-element $g\rho \in \GL(V)$ is unipotent and hence fixes a nonzero vector of $V$. However, $g\gamma \in T^k$ has order $p$, so $g\gamma$ is contained in no $Q$-conjugate of $M^k$ as $p$ does not divide $|M^k|$. Therefore, $g\gamma$ is contained in no $G\gamma$-conjugate of $H\gamma$, and hence $g$ is contained in no $G$-conjugate of $H$. Said otherwise, $g$ is derangement in the action of $G$ on $G/H$, which is a contradiction. Therefore, $p$ divides $|M|$ as claimed. 

Finally, note that $Q = G\gamma = (G_0H)\gamma = T^k (H\gamma)$, so $\bigcup_{g \in Q}(H\gamma)^g = \bigcup_{g \in T^k}(H\gamma)^g$. By intersecting with $T^k$, we see that $\bigcup_{g \in Q}(M^k)^g = \bigcup_{g \in T^k}(M^k)^g$.

We have now verified that $C$ is a counterexample, as claimed.

Therefore, Lemma~\ref{lem:isbell_reduction} implies that there exists a counterexample 
\[
C_0 = (T,M,k_0,T^{k_0},G_0,N_0,\gamma_0,d_0,\rho_0,H_0)
\] 
where $N_0 \leq \Frat(G_0)$ and $\rho_0\: G_0 \to \GL_{d_0}(p)$ is primitive. From this we obtain a faithful primitive representation $\rho_1\: G_0 \to \GL_n(\FF_p)$, for a divisor $n$ of $d_0$ (see \cite[Lemma~2.10.2]{ref:KleidmanLiebeck}). 

Let us record some consequences of $C_0$ being a counterexample to Hypothesis~\ref{hyp:isbell_reduction}. First, there is no element $g \in G_0$ such that $g$ is a derangement in the action of $G_0$ on $G_0/H_0$ and such that $g\rho_0$ fixes a nonzero vector of $\F_p^{d_0}$. Second, $p^{d_0}$ divides $|T:M|^{k_0}$ and $p$ divides $|M|$. In particular,
\[
d_0 \leq k_0 \cdot v_p(|T:M|) < k_0 \cdot v_p(|T|).
\] 

We claim that $n < 2^{n_T'} \cdot k_0$ if $p \neq 2$. Since $p^{d_0}$ divides $|G_0:H_0|$ and $p$ divides $|M_0|$,
\[
n \leq d_0 \leq v_p(|G_0:H_0|) < v_p(|G_0:N_0|) = v_p(|T^{k_0}|) = v_p(|T|) \cdot k_0.
\] 
If $p \neq 2$, then Proposition~\ref{prop:p-part}(ii) implies that $v_p(|T|) \leq 2^{n_T'}$, so $n < 2^{n_T'} \cdot k_0$, as claimed.

We claim that $N_0 = Z(G_0)$. Write $\overline{X} = XZ(G_0)/Z(G_0)$ for $X \leq G_0$. Let $\lambda\: \overline{G_0} \to \PGL_n(\FF_p)$ be the faithful primitive projective representation corresponding to $\rho_1\: G_0 \to \GL_n(\FF_p)$. Since $Z(T^{k_0}) = 1$, it follows that $Z(G_0) \leq N_0$, and let $\overline{\gamma}\: \overline{G_0} \to T^{k_0}$ factorise $\gamma\:G_0 \to T^{k_0}$. Note that $\ker\overline{\gamma} = \overline{N_0} \leq \Frat(\overline{G_0})$ since $N_0 \leq \Frat(G_0)$, so $K\overline{\gamma} < \overline{G_0}\overline{\gamma} = T^{k_0}$ for all $K < \overline{G_0}$. Applying \cite[Theorem~4]{ref:HarperLiebeck25}, we deduce that $\overline{N_0} = \ker\lambda = 1$, or said otherwise, $N_0 = Z(G_0)$, as claimed.

Therefore, $G_0/Z(G_0) = T^{k_0}$ and $Z(G_0) \leq H_0 \leq G_0$ satisfies $H_0/Z(G_0) = M^{k_0}$. If $G_0' < G_0$, then $T^{k_0} = (T^{k_0})' = G_0'Z(G_0)/Z(G_0) < T^{k_0}$, which is absurd, so $G_0' = G_0$. Said otherwise, $G_0$ is a perfect central extension $T^{k_0}$.

Let $V$ be the $\F_pG_0$-module afforded by the faithful primitive representation $\rho_0$, and let $E = \End_{\F_pG_0}(V)$. Write $e = |E:\F_p|$, so $e$ divides $d_0$ and we may view $G_0$ as an irreducible subgroup of $\GL_{d_0/e}(p^e)$ (see \cite[Lemma~2.10.2]{ref:KleidmanLiebeck}). 

We claim that $k = 1$. Suppose otherwise. Writing $T^k = T_1 \times \cdots \times T_k$, for each $1 \leq i \leq k$, fix $Z(G_0) \leq G_i \leq G_0$ such that $G_i/Z(G_0) = T_i$. Then $G_0$ is the central product $G_1 \circ \cdots \circ G_k$. Therefore, by \cite[Lemma~5.5.5]{ref:KleidmanLiebeck}, $G_1$ is an irreducible subgroup of $\GL_t(p^e)$ where $t^{k_0} = d_0/e$. Then
\[
v_p(|T|) = v_p(|G_1|) \leq v_p(|\GL_t(p^e)|) = et(t-1)/2 < et^2/2 \leq et^{k_0}/{k_0} = d_0/k_0,
\]
but this contradicts the fact that $d_0 \leq k_0 \cdot v_p(|T|)$. This proves the claim that $k=1$.

We claim that $T$ is a group of Lie type in characteristic $p$. Since $G_0/Z(G_0) = T$ is an absolutely irreducible subgroup of $\PGL_{d_0/e}(p^e)$, we have
\[
v_p(|T|) > v_p(|T:M|) \geq d_0 \geq d_0/e \geq R_p(T).
\]
By Proposition~\ref{prop:p-part}(iii), this implies that $T$ is a group of Lie type in characteristic $p$ or $p=2$ and $T \in \mathcal{E} = \{ A_8, \, \PSU_4(3), \, {\rm M}_{22}, \, {\rm J}_2, \, {\rm Suz} \}$. Therefore, it remains to exclude the possibility that $p=2$ and $T \in \mathcal{E}$. First assume that $T \in \{ \PSU_4(3), \, {\rm M}_{22}, \, {\rm J}_2, \, {\rm Suz} \}$. Consulting \cite[Proposition~5.3.8]{ref:KleidmanLiebeck}, $R_2(T)= v_2(T) - 1$. However,
\[
R_2(T) \leq d_0 \leq v_2(|T:M|) = v_2(|T|) - v_2(|M|),
\]
so $v_2(|M|) = 1$, but this is impossible (see \cite{ref:ATLAS}). It remains to assume that $T = A_8$. Consulting \cite[Proposition~5.3.7]{ref:KleidmanLiebeck}, $R_2(S) = 4 = v_2(S) - 2$. Therefore, $v_2(|M|) \leq 2$, but again this is impossible (the maximal subgroups of $A_8$ and $S_8$ can be obtained in \textsc{Magma} \cite{ref:Magma}, for example). This proves the claim that $T$ is a group of Lie type in characteristic $p$.

The proof is now completed by Proposition~\ref{prop:isbell_simple}.
\end{proof}

\section{Proof of Theorem~\ref{thm:main}} \label{s:proofs}

Before proving Theorem~\ref{thm:main}, it will be convenient to first prove the following special case.

\begin{shtheoremstar}
Let $G \leq \Sym(\Omega)$ have exactly two orbits $\Omega_1$ and $\Omega_2$, both of which are nontrivial. Assume that $|\Omega_1|$ divides $|\Omega_2|$ and $G$ acts faithfully and primitively on $\Omega_1$ and  primitively on $\Omega_2$. Then $G$ has a derangement.
\end{shtheoremstar}

More precisely, we focus on almost simple groups in Section~\ref{ss:proofs_as} and affine groups in Section~\ref{ss:proofs_affine}, before completing the proof of Theorem~\ref{thm:main}* in Section~\ref{ss:proofs_proof}. Section~\ref{ss:proofs_proof} also includes proofs that deduce Theorem~\ref{thm:main} from Theorem~\ref{thm:main}* and Corollary~\ref{cor:prime_power} from Theorem~\ref{thm:main}.

\subsection{Theorem~\ref{thm:main}* for almost simple groups} \label{ss:proofs_as}

This section is dedicated to proving Theorem~\ref{thm:main}* for almost simple groups. We use the study of normal coverings of almost simple groups due to Bubboloni, Spiga and Weigel \cite{ref:BubboloniSpigaWeigel}. In the following two proofs, we write $\pi(n)$ for the set of prime divisors of a positive integer $n$.

\begin{proposition} \label{prop:as}
Let $G$ be an almost simple group with socle $T$. Let $H_1$ and $H_2$ be core-free subgroups of $G$ such that $|H_2|$ divides $|H_1|$. Then $T \neq \bigcup_{g \in G} (H_1 \cap T)^g \cup \bigcup_{g \in G} (H_2 \cap T)^g$.
\end{proposition}

\begin{proof}
For a contradiction, suppose otherwise. In particular, $T = \bigcup_{1 \leq i \leq 2} \bigcup_{a \in \Aut(T)} (H_i \cap T)^a$, so, by definition, $\{H_1 \cap T, H_2 \cap T\}$ is a \emph{weak normal $2$-covering} of $T$. By \cite[Theorem~1.5]{ref:BubboloniSpigaWeigel}, the weak normal $2$-coverings of finite simple groups are given in \cite[Tables~3--7]{ref:BubboloniSpigaWeigel} and none of the examples are compatible with $|H_2|$ dividing $|H_1|$. This check can be sped up by noting that since $\pi(|T|) \subseteq \pi(|H_1|) \cup \pi(|H_2|)$ and $|H_2|$ divides $|H_1|$, we have $\pi(|T|) \subseteq \pi(|H_1|)$, so, by \cite[Corollary~5]{ref:LiebeckPraegerSaxl00}, the only possibilities for $(G,H_1)$ are given in \cite[Table~10.7]{ref:LiebeckPraegerSaxl00}. 
\end{proof}

\begin{remark} \label{rem:as}
By Remark~\ref{rem:coverings}, Proposition~\ref{prop:as} establishes Conjecture~\ref{conj:main} for simple groups.
\end{remark}

We can now establish the main result on almost simple groups. 

\begin{proposition} \label{prop:technical_as}
Let $G \leq \Sym(\Omega)$ have exactly two orbits $\Omega_1$ and $\Omega_2$, both of which are nontrivial. Assume that $|\Omega_1|$ divides $|\Omega_2|$ and $G$ acts faithfully and primitively on $\Omega_1$ and  primitively on $\Omega_2$. Assume that $G$ is almost simple. Then $G$ has a derangement.
\end{proposition}

\begin{proof}
For a contradiction, suppose that $G$ has no derangement on $\Omega$. Let $H_1$ and $H_2$ be stabilisers of points in $\Omega_1$ and $\Omega_2$, respectively, so $H_1$ is core-free as $G$ is faithful on $\Omega_1$. Then $G$ is the union of conjugates of $H_1$ and $H_2$, so $\pi(|G|) = \pi(|H_1|) \cup \pi(|H_2|)$. Moreover, $|H_2|$ divides $|H_1|$, so $\pi(|G|) = \pi(|H_1|)$ and $(G,H_1)$ appears in \cite[Table~10.7]{ref:LiebeckPraegerSaxl00}. Writing $T = \soc(G)$, we note that $|G:H_1| > |\Out(T)|$, so $|H_2| \leq |H_1| < |T|$ and $H_2$ is also core-free. However, $T = \bigcup_{g \in G} (H_1 \cap T)^g \cup \bigcup_{g \in G} (H_2 \cap T)^g$, which contradicts Proposition~\ref{prop:as}. 
\end{proof}

\subsection{Theorem~\ref{thm:main}* for affine groups} \label{ss:proofs_affine}

This section is dedicated to proving Theorem~\ref{thm:main}* when the faithful primitive action is affine. We begin with some preliminary lemmas.

\begin{lemma} \label{lem:affine}
Let $G = V{:}H$ be a finite primitive affine group. Let $K$ be a maximal subgroup of $G$. Then either $K$ is a complement of $V$ or $K = V{:}M$ for a maximal subgroup $M$ of $H$.
\end{lemma}

\begin{proof}
Let $\p\:G \to H$ be the projection onto $H$ with kernel $V$. First assume that $K\p < H$. Then $K \leq V{:}(K\p)$, so by the maximality of $K$, we deduce that $K = V{:}(K\p)$, where $K\p$ is maximal in $H$. Now assume that $K\p = H$. Suppose that $V \cap K$ is nontrivial, and fix $0 \neq v \in V \cap K$. Since $K\p = H$ and $H$ acts irreducibly on $V$, we deduce that $\< v^K \> = V$, so $V \leq K$, but this implies that $K = G$, which is absurd. Therefore, $V \cap K = 1$, so $K$ is a complement of $G$.
\end{proof}

\begin{lemma} \label{lem:affine_derangement}
Let $M \leq H \leq \GL(V)$ and write $G = V{:}H$. Let $h \in H$ and $v \in V$. Assume that $h$ is a derangement in the action of $H$ on $H/M$ and that $v \not\in \im(h-1)$. Then $(v,h^{-1}) \in G$ is contained in no $G$-conjugate of $H$ and no $G$-conjugate of $V{:}M$.
\end{lemma}

\begin{proof}
Note that $h^{-1} \not\in \bigcup_{g \in H} M^g$, so $(v,h^{-1}) \not\in \bigcup_{g \in G} (V{:}M)^g$. Suppose that $(v,h^{-1}) \in H^g$ for some $g \in G$. Then $H^g$ is the stabiliser of a vector $u \in V$, so $u = u(v,h^{-1}) = (u+v)h^{-1}$. This means that $u(h-1) = v$, which contradicts the hypothesis that $v \not\in \im(h-1)$.
\end{proof}

We can now establish the main result on affine groups. This is where we use Theorem~\ref{thm:isbell}.

\begin{proposition} \label{prop:technical_affine}
Let $G \leq \Sym(\Omega)$ have exactly two orbits $\Omega_1$ and $\Omega_2$, both of which are nontrivial. Assume that $|\Omega_1|$ divides $|\Omega_2|$. Assume that $G = V{:}H_1$ is a faithful primitive affine group on $\Omega_1 = V$, and assume that $G$ acts primitively on $\Omega_2$. Then $G$ has a derangement.
\end{proposition}

\begin{proof}
Note that $H_1$ is the stabiliser of a point of $\Omega_1$, and let $H_2$ be the stabiliser of a point of $\Omega_2$, so $H_2$ is maximal. First assume that $H_2$ is a complement of $V$. Then $H_1 \cap V = H_2 \cap V = 1$, so $\{H_1,H_2\}$ is not a normal covering of $G$, so $G$ has a derangement on $\Omega$. By Lemma~\ref{lem:affine}, we may now assume that $H_2 = V{:}M$ for a maximal subgroup $M$ of $H_1$. Since $|\Omega_1|$ divides $|\Omega_2|$, we deduce that $|G:H_1|=|V|$ divides $|G:H_2|=|H_1:M|$. Write $|V| = p^d$. By Theorem~\ref{thm:isbell}, there exists $h \in H_1$ that fixes a nonzero vector in $V$ and is a derangement in the action of $H_1$ on $H_1/M$. In particular, $\ker(h-1) \neq 0$, so $\im(h-1)$ is a proper subspace of $V$. Fix $v \in V$ such that $v \not\in \im(h-1)$. Then $(v,h^{-1})$ is a derangement on $\Omega$ by Lemma~\ref{lem:affine_derangement}.
\end{proof}

\subsection{Proof of Theorems~\ref{thm:main}* and~\ref{thm:main}} \label{ss:proofs_proof}

Before proving Theorem~\ref{thm:main}* we consider a final special case.

\begin{proposition} \label{prop:technical_faithful}
Let $G \leq \Sym(\Omega)$ have exactly two orbits $\Omega_1$ and $\Omega_2$, both of which are nontrivial. Assume that $|\Omega_1|$ divides $|\Omega_2|$ and $G$ acts faithfully and primitively on $\Omega_1$ and $\Omega_2$. Then $G$ has a derangement.
\end{proposition}

\begin{proof}
Let $H_1$ and $H_2$ be stabilisers of points in $\Omega_1$ and $\Omega_2$, respectively.

First assume that $G$ is almost simple and suppose that $G$ does not have a derangement. Then $G = \bigcup_{g \in G} H_1^g \cup \bigcup_{g \in G} H_2^g$, which contradicts Proposition~\ref{prop:as}. For the remainder of the proof we may assume that $G$ is not almost simple.

Next assume that $G$ has a normal subgroup $N$ that is regular on $\Omega_2$. Since $G$ is faithful and primitive on $\Omega_1$, the nontrivial normal subgroup $N$ acts transitively on $\Omega_1$. Therefore, there exists $g \in N$ that is a derangement on $\Omega_1$. However, $g$ is a derangement on $\Omega_2$ since $N$ is regular on $\Omega_2$, so $G$ is a derangement on $\Omega$. A symmetric argument handles the case where $G$ has a normal subgroup that is regular on $\Omega_1$. Therefore, for the remainder of the proof we may assume that $G$ has no normal subgroup that is regular on $\Omega_1$ or $\Omega_2$. 

Let $T = \soc(G)$. By Remark~\ref{rem:o'nan-scott}\ref{rem:o'nan-scott_regular}, $T$ is a minimal normal subgroup and there exists a nonabelian simple group $S$ and an integer $k \geq 1$ such that $T = S_1 \times \cdots \times S_k$ where $S_i \cong S$ for each $1 \leq i \leq k$. Since $G$ is not almost simple, $k \geq 2$. Since $T$ is not regular on $\Omega_1$ or $\Omega_2$, by Remark~\ref{rem:o'nan-scott}\ref{rem:o'nan-scott_rest}, $G$ does not have type (TW) on $\Omega_1$ or $\Omega_2$.

Since $T$ is a nontrivial normal subgroup, $T$ is transitive on $\Omega_1$ and $\Omega_2$, so $G = TH_1 = TH_2$, so, by conjugation, $H_1$ and $H_2$ both transitively permute the $k$ simple factors of $T$ since $T$ is a minimal normal subgroup of $G$. Hence, there exist $R_1, R_2 \leq S$ such that for all $1 \leq i \leq k$, the projections of $H_1 \cap T$ and $H_2 \cap T$ onto $S_i$ are $R_1$ and $R_2$ respectively. 

Let $i \in \{1,2\}$. We will make two observations.

Consider the case $R_i < S$. Since $G$ does not have type (TW) on $\Omega_i$, we know that $G$ has type (PA) on $\Omega_i$. In particular, there exists an almost simple group $S \leq A \leq \Aut(S)$ and a maximal subgroup $M_i$ of $A$ such that $H_i \cap T = (M_i \cap S)^k$.

Consider the case $R_i = S$, so $G$ has type (SD) or (CD) on $\Omega_i$. In this case, we can partition $I = \{1, \dots, k\}$ as $\{I_{i1}, \dots, I_{il_i} \}$ where $|I_{ij}| = k/l_i$ and write $H_i \cap T = D_{i1} \times \cdots \times D_{il_i}$ where $D_{ij} = \{ (x^{\alpha_{ij1}}, \dots, x^{\alpha_{ijl_i}}) \mid x \in S \} \leq \prod_{r \in I_{ij}} S_r$ for $\alpha_{ijr} \in \Aut(S)$. In particular, $S_1 \cap H_i = 1$. 

With these observations in place, we can complete the proof by dividing into three cases.

First assume that both $R_1$ and $R_2$ equal $S$. Let $g \in S_1$ be nontrivial. Then $g$ is contained in no conjugate of $H_1$ or $H_2$, so $g$ is a derangement on $\Omega = \Omega_1 \cup \Omega_2$.

Next assume that exactly one of $R_1$ and $R_2$ is equal to $S$. Without loss of generality we may assume that $R_2 = S$. By \cite[Proposition~2]{ref:Saxl88}, there exists $g \in S_1$ such that, viewing $g$ as an element of $S$, we have $g \not\in \bigcup_{a \in A}(M_1 \cap S)^a$, so $g$ is contained in no conjugate of $H_1$, and $g$ is contained in no conjugate of $H_2$ since $g \in S_1$. Therefore, $g$ is a derangement on $\Omega$.

Finally assume that neither $R_1$ nor $R_2$ is equal to $S$. For each $i \in \{1,2\}$, by applying \cite[Proposition~2]{ref:Saxl88}, there exists $g_i \in S_1$ such that, viewing $g_i$ as an element of $S$, we have $g_i \not\in \bigcup_{a \in A}(M_i \cap S)^a$. Hence, $g = (g_1,g_2,1,\dots,1) \in T$ is in no conjugate of $H_1$ or $H_2$, so $g$ is a derangement on $\Omega$.
\end{proof}

\begin{proof}[Proof of Theorem~\ref{thm:main}*]
Let $K_2$ be the kernel of the action of $G$ on $\Omega_2$. By Propositions~\ref{prop:technical_as} and~\ref{prop:technical_affine}, we can assume that the action of $G$ on $\Omega_1$ is not affine or almost simple, and by Proposition~\ref{prop:technical_faithful}, we can assume that $K_2$ is nontrivial. Let $T = \soc(G)$, so $T \cong S^k$ where $S$ is a nonabelian finite simple group and $k \geq 2$. In particular, $G \leq \Aut(S) \wr S_k$. 

First assume that $G$ has a minimal normal subgroup $N$ that acts nontrivially on $\Omega_2$. If $G$ has a unique minimal normal subgroup, then $N \leq K_2$, which is impossible since $K_2$ is trivial on $\Omega_2$. Therefore, $G$ has exactly two minimal normal subgroups, which are necessarily regular on $\Omega_1$ (see Remark~\ref{rem:o'nan-scott}(iii)). Since $N$ is a normal subgroup of $G$ acting nontrivially on $\Omega_2$ and $G$ acts primitively on $\Omega_2$, we know that $N$ acts transitively on $\Omega_2$. Therefore, fix $g \in N$ that is a derangement on $\Omega_2$. Now $g$ is nontrivial on $\Omega_1$ since $G$ is faithful on $\Omega_1$, but $N$ acts regularly on $\Omega_1$, so $g$ is a derangement on $\Omega_1$. Hence, $g$ is a derangement on $\Omega$.

It now remains to exclude the possibility that $T$ acts trivially on $\Omega_2$. For a contradiction, suppose that $T$ does act trivially on $\Omega_2$, so $T \leq K_2$. Write $\overline{G} = G/K_2$ and note that $\overline{G}$ has a faithful primitive action on $\Omega_2$. 

We claim that $\soc(\overline{G})$ is isomorphic to a quotient of a subgroup of $S_k$. To prove this, we first need to introduce some notation. Note that $\overline{G} = (G/T)/(K_2/T)$. Write $X = G/T$ and let $\pi\: X \to \overline{G}$ be the quotient map defined by $K_2/T$. Let $Y \leq X$ be the full preimage under $\pi$ of $\soc(\overline{G})$. Let $W = \Out(S) \wr S_k$, and let $B = \Out(S)^k$ be the base group of $W$, which is a soluble normal subgroup of $W$. Note that $Y \leq X \leq W$. Therefore, if $(Y \cap B)\pi$ is trivial, then $Y \cap B \leq \ker\pi$, so $\soc(\overline{G}) = Y\pi$ is a quotient of $Y/(Y \cap B) \cong YB/B \preccurlyeq S_k$, as required. Hence, to prove the claim, it suffices to prove that $(Y \cap B)\pi$ is trivial. To do this, we divide into two cases.

First assume that $\overline{G}$ is not affine on $\Omega_2$. Then $\soc(\overline{G})$ is a nonabelian characteristically simple group, so, in particular, has no nontrivial soluble normal subgroups. However, $Y \cap B$ is a soluble normal subgroup of $Y$, so $(Y \cap B)\pi$ is a soluble normal subgroup of $Y\pi = \soc(\overline{G})$, which implies that $(Y \cap B)\pi$ is trivial.

Now assume that $\overline{G}$ is affine on $\Omega_2$. Then $\soc(\overline{G})$ is an elementary abelian $p$-group and $|\Omega_2|$ is a power of $p$. Since $|\Omega_1|$ divides $|\Omega_2|$, the degree $|\Omega_1|$ is a power of $p$ too. Since we have assumed that the action of $G$ on $\Omega_1$ is not affine or almost simple, the action of $G$ on $\Omega_1$ must have type (PA) (see Remark~\ref{rem:o'nan-scott}\ref{rem:o'nan-scott_degree} for example). In particular, $S$ has a proper subgroup whose index is a power of $p$. The possibilities for such $S$ and $p$ are given in \cite[Theorem~1]{ref:Guralnick83}, and in each case it is easy to see that $p$ does not divide $|\!\Out(S)|$. Therefore, $\soc(\overline{G})$ and $B \cong \Out(S)^k$ have coprime orders, so $Y\pi = \soc(\overline{G})$ and $B\pi$ have coprime orders. This means that $Y\pi \cap B\pi$ is trivial, but $(Y \cap B)\pi \leq Y\pi \cap B\pi$, so $(Y \cap B)\pi$ is trivial too.

We have proved the claim that $\soc(\overline{G})$ is isomorphic to a quotient of a subgroup of $S_k$.

Since $\soc(\overline{G})$ is a nontrivial normal subgroup of $\overline{G}$ and $\overline{G}$ acts primitively on $\Omega_2$, we know that $\soc(\overline{G})$ acts transitively on $\Omega_2$, so in particular, $|\Omega_2|$ divides $|\!\soc(\overline{G})|$. However, $|\Omega_1|$ divides $|\Omega_2|$ and $\soc(\overline{G})$ is isomorphic to a quotient of a subgroup of $S_k$, so we conclude that $|\Omega_1|$ divides $k!$. 

We claim that there exists a prime $p$ such that $p^k$ divides $|\Omega_1|$. If $G$ has type (PA) on $\Omega_1$, then $|\Omega_1| = d^k$ where $d > 1$ divides $|S|$, so the claim holds. Otherwise, by Remark~\ref{rem:o'nan-scott}\ref{rem:o'nan-scott_degree} for example, $|S|^{k/2}$ divides $|\Omega_1|$, and since $S$ is a nonabelian simple group, there exists a prime $p$ such that $p^2$ divides $|S|$, so $p^k$ divides $|\Omega_1|$, which proves the claim. (Indeed, one can choose $p$ to be the least prime divisor of $|S|$. If $p^2$ does not divide $|S|$, then a Sylow $p$-subgroup has order $p$, so, as a consequence of Burnside's normal $p$-complement theorem (see \cite[(39.2)]{ref:Aschbacher00}), it has a normal complement, which contradicts $S$ being a nonabelian simple group.) 

The greatest power of $p$ that divides $k!$ is $p^m$ where $m = (k-s_p(k))/(p-1)$ where $s_p(k)$ is the sum of the digits in the base $p$ expansion of $k$. In particular, $m < k$, which contradicts $p^k$ dividing $|\Omega_1|$ and $|\Omega_1|$ dividing $k!$. This completes the proof.
\end{proof}

We now deduce Theorem~\ref{thm:main} from Theorem~\ref{thm:main}*.

\begin{proof}[Proof of Theorem~\ref{thm:main}]
Let $H_1$ and $H_2$ be stabilisers of points in $\Omega_1$ and $\Omega_2$, respectively. Let $M_1$ be a maximal subgroup of $G$ that contains $H_1$, and let $M_2 = H_2$, which is also maximal. Let $N_1$ and $N_2$ be the kernels of the primitive actions of $G$ on $\Delta_1 = G/M_1$ and $\Delta_2 = G/M_2$, respectively. Write $\Delta = \Delta_1 \cup \Delta_2$, noting that $|\Delta_1|$ divides $|\Delta_2|$. The kernel of the action of $G$ on $\Delta$ is $N = N_1 \cap N_2$. Write $\overline{X} = XN/N$ for $X \leq G$. An element of $\overline{G}$ that is a derangement on $\Delta$ corresponds to an element of $G$ that is a derangement on $\Delta$ and hence $\Omega$, so it suffices to prove that $\overline{G}$ has a derangement on $\Delta$. If $\overline{N}_1$ acts transitively on $\Delta_2$, then Lemma~\ref{lem:technical} implies that $G$ has a derangement. Now assume that $\overline{N}_1$ does not act transitively on $\Delta_2$. Since $\overline{G}$ acts primitively on $\Delta_2$, we know that $\overline{N}_1$ acts trivially on $\Delta_2$, but $\overline{N}_1$ acts trivially on $\Delta_1$, so $\overline{N}_1$ acts trivially on $\Delta$, so $\overline{N}_1 = 1$. Theorem~\ref{thm:main}* now implies that $\overline{G}$ has a derangement.
\end{proof}

All that remains is to prove Corollary~\ref{cor:prime_power}.

\begin{proof}[Proof of Corollary~\ref{cor:prime_power}]
Let $H_1$ and $H_2$ be stabilisers of points in $\Omega_1$ and $\Omega_2$, respectively. 
For $i \in \{1,2\}$, let $M_i$ be a maximal subgroup of $G$ containing $H_i$ and let $N_i$ be the kernel of the primitive action of $G$ on $\Delta_i = G/M_i$. As in the previous proof, $N = N_1 \cap N_2$ is the kernel of the action on $\Delta = \Delta_1 \cup \Delta_2$ and it suffices to prove that $G/N$ has a derangement on $\Delta$. By interchanging the roles of $\Omega_1$ and $\Omega_2$ if necessary, we may assume that $|\Delta_1|$ divides $|\Delta_2|$. Therefore, Theorem~\ref{thm:main} implies that $G/N$ has a derangement on $\Delta$.
\end{proof}

\section{Proof of Theorems~\ref{thm:summary} and~\ref{thm:reduction}} \label{s:reduction}

We conclude the paper by turning our attention to Theorems~\ref{thm:summary} and~\ref{thm:reduction}. The main technical result in this direction is the following, and we thank Pablo Spiga for sharing the proof of this with us.

\begin{theorem} \label{thm:reduction_technical}
Let $\C$ be a class of finite groups closed under taking normal subgroups.
Assume that Conjecture~\ref{conj:main} holds for all perfect groups in $\C$.
Then Conjecture~\ref{conj:main} holds for all groups in $\C$.
\end{theorem}

\begin{proof}
Let $G \leq \Sym(n)$ be a group in $\C$ and assume that $G$ has two orbits of size $\frac{n}{2} > 1$.
Let $H_1$ and $H_2$ be stabilisers in $G$ of points in different orbits. It suffices to prove that 
\[
G \neq \bigcup_{g \in G}H_1^g \cup \bigcup_{g \in G}H_2^g.
\]
We proceed by induction on the composition length of $G$. For the base case, assume that $G$ is simple. If $G$ is perfect, then the result holds by hypothesis. Therefore, we may assume that $G$ has prime order. This forces $H_1 = H_2 = 1$, so
\[
\bigcup_{g \in G}H_1^g \cup \bigcup_{g \in G}H_2^g = 1 < G.
\]

For the inductive step, assume that the composition length of $G$ is $c > 1$. If $G$ is perfect, then the result holds by hypothesis. Therefore, we may assume that $G$ has a maximal normal subgroup $N$ such that $G/N$ has prime order.

First assume that $G = H_1N = H_2N$. Then 
\[
|N:H_1 \cap N| = |H_1N:H_1| = |G:H_1| = |G:H_2| = |H_2N:H_2| = |N:H_2 \cap N|.
\]
The class $\C$ is closed under taking normal subgroups, so $N$ is also contained in $\C$.
Since $N$ has composition length $c-1$, by induction, 
\[
N \neq \bigcup_{g \in N}(H_1 \cap N)^g \cup \bigcup_{g \in N}(H_2 \cap N)^g = \left(\bigcup_{g \in N}H_1^g \cup \bigcup_{g \in N}H_2^g\right) \cap N.
\]
In particular, 
\[
\bigcup_{g \in N}H_1^g \cup \bigcup_{g \in N}H_2^g \neq G.
\]
Therefore, since $G = H_1N = H_2N$, we have
\[
\bigcup_{g \in G}H_1^g \cup \bigcup_{g \in G}H_2^g = \bigcup_{g \in N}H_1^g \cup \bigcup_{g \in N}H_2^g \neq G.
\]

It remains to assume that $H_1N < G$ or $H_2N < G$. If $H_1N < G$ and $H_2N < G$, then since $G/N$ has prime order, $H_1 \leq N$ and $H_2 \leq N$, so
\[
\bigcup_{g \in G}H_1^g \cup \bigcup_{g \in G}H_2^g \leq N < G.
\]
Therefore, without loss of generality, we may assume that $H_1N < G$ and $H_2N = G$. For a contradiction, suppose that 
\[
G = \bigcup_{g \in G}H_1^g \cup \bigcup_{g \in G}H_2^g.
\]
Since $G/N$ has prime order, $H_1 \leq N$, which implies that
\[
G = N \cup \bigcup_{g \in G}H_2^g.
\]
In particular, every element of $G$ that is a derangement on $G/H_2$ is contained in $N$. This means that $D(G) \leq N$, where $D(G)$ is the subgroup of $G$ generated by the elements of $G$ that are derangements on $G/H_2$. In particular, $|G:N|$ also divides $|G:D(G)|$, and, as noted in \cite[Corollary~1.2]{ref:BaileyCameronGiudiciRoyle21}, $|G:D(G)|$ divides $|G:H_2|-1$, so $|G:N|$ divides $|G:H_2|-1$. However, $H_1 \leq N$, so $|G:N|$ divides $|G:H_1| = |G:H_2|$, which is a contradiction.
\end{proof}

We record two immediate consequences of Theorem~\ref{thm:reduction_technical}.

\begin{corollary} \label{cor:soluble}
Conjecture~\ref{conj:main} holds for soluble groups.
\end{corollary}

\begin{proof}
Apply Theorem~\ref{thm:reduction_technical} with $\C$ as the class of soluble groups, noting that there are no nontrivial perfect soluble groups.
\end{proof}

\begin{corollary} \label{cor:reduction}
For all integers $m \geq 1$, if Conjecture~\ref{conj:main} holds for all perfect groups of order at most $m$, then Conjecture~\ref{conj:main} holds for all groups of order at most $m$. In particular, to prove Conjecture~\ref{conj:main}, it suffices to assume that $G$ is perfect.
\end{corollary}

\begin{proof}
Apply Theorem~\ref{thm:reduction_technical} with $\C$ as the class of finite groups of order at most $m$.
\end{proof}

Theorem~\ref{thm:reduction} is an immediate consequence of Corollary~\ref{cor:reduction}. It remains to prove Theorem~\ref{thm:summary}.

\begin{proof}[Proof of Theorem~\ref{thm:summary}]
Part~(i) is Corollary~\ref{cor:soluble}. Applying Theorem~\ref{thm:reduction_technical} with $\C$ as the class of almost simple groups reduces part~(ii) to the case where $G$ is simple (since the outer automorphism group of a simple group is soluble), and Conjecture~\ref{conj:main} holds for simple groups by Remark~\ref{rem:as}. Applying, Corollary~\ref{cor:reduction} reduces part~(iii) to the perfect groups of order at most 50000, but \textsc{Magma} \cite{ref:Magma} has a database of perfect groups of order at most 50000 and a straightforward computation verifies that Conjecture~\ref{conj:main} holds for all of them (our code is given in \cite{ref:HarperGH2024}).
\end{proof}

\vspace{11pt}

\begin{multicols}{2}
\noindent David Ellis \newline
School of Mathematics \newline
University of Bristol \newline
Bristol, BS8 1UG, UK \newline
\texttt{david.ellis@bristol.ac.uk}

\noindent Scott Harper \newline
School of Mathematics  \newline
University of Birmingham \newline
Birmingham, B15 2TT, UK \newline
\texttt{s.harper.3@bham.ac.uk}
\end{multicols}

\end{document}